\documentclass[
	english
]{scrartcl}

\usepackage[utf8]{inputenc}

\usepackage{amsmath,amsfonts,amsthm,amssymb}
\usepackage{hyperref}
\hypersetup{colorlinks,citecolor=blue}
\usepackage{cancel}
\usepackage{relsize}
\usepackage{float}
\usepackage{algpseudocode}
\usepackage{changes}
\usepackage{enumerate}
\usepackage{doi}
\usepackage{enumitem}
\setlist[enumerate]{label=(\alph*)}
\usepackage{pdfrender}
\usepackage{booktabs}

\usepackage[figure]{hypcap}
\usepackage{graphicx}
\graphicspath{{./img/}}
\usepackage{subcaption} 

\usepackage{preamble}

\usepackage{marginnote}

\usepackage[capitalize,noabbrev,nameinlink]{cleveref} 
\allowdisplaybreaks

\numberwithin{equation}{section}

\newcommand\norm[1]{\left\Vert#1\right\Vert}

\newcommand\abs[1]{\left\vert#1\right\vert}
\newcommand\dual[2]{\left\langle #1, #2\right\rangle}

\newcommand\N{\mathbb{N}}
\newcommand\R{\mathbb{R}}

\newcommand\LL{\mathcal L}
\newcommand\XX{\mathcal X}
\newcommand\YY{\mathcal Y}
\newcommand\ZZ{\mathcal Z}
\newcommand\UU{\mathcal U}
\newcommand\VV{\mathcal V}
\renewcommand\AA{\mathcal A}
\newcommand\II{\mathcal I}

\newcommand{\proj}{\operatorname{proj}}

\newcommand{\ad}{\textup{ad}}

\DeclareMathOperator*{\stt}{\operatorname{s.t.}}

\DeclareMathAlphabet{\mathpzc}{OT1}{pzc}{m}{it}
\newcommand\oo{\mathpzc{o}}

\theoremstyle{plain}
	\newtheorem{theorem}{Theorem}[section]
	
	\newtheorem{lemma}[theorem]{Lemmas}
	\newtheorem{algorithm}[theorem]{Algorithm}
	\newtheorem{proposition}[theorem]{Proposition}
	\newtheorem{assumption}[theorem]{Assumption}

	\newtheorem{remark}[theorem]{Remark}
	
\crefname{assumption}{Assumption}{Assumptions}
\crefname{figure}{Figure}{Figures}

\AddToHook{env/lemma/begin}{\crefalias{theorem}{lemma}}
\AddToHook{env/proposition/begin}{\crefalias{theorem}{proposition}}
\AddToHook{env/algorithm/begin}{\crefalias{theorem}{algorithm}}
\AddToHook{env/assumption/begin}{\crefalias{theorem}{assumption}}
\AddToHook{env/remark/begin}{\crefalias{theorem}{remark}}
\AddToHook{env/example/begin}{\crefalias{theorem}{example}}
\AddToHook{env/corollary/begin}{\crefalias{theorem}{corollary}}
\AddToHook{env/definition/begin}{\crefalias{theorem}{definition}}

\counterwithin{figure}{section}
\counterwithin{table}{section}

\makeatletter
\long\def\@firstoffiveparen#1#2#3#4#5{\textup{\tagform@{#1}}}
\def\eqref@nolink#1{\textup{\tagform@{\ref*{#1}}}}
\def\eqref@link#1{%
\Hy@safe@activestrue
\expandafter\@setref\csname r@#1\endcsname\@firstoffiveparen{#1}%
\Hy@safe@activesfalse
}
\protected\def\eqref{\@ifstar\eqref@nolink\eqref@link}
\makeatother

\newif\ifshowcomments
\showcommentstrue 

\definecolor{mygreen}{rgb}{0.0,0.7,0.0}
\definecolor{mybrown}{rgb}{0.5,0.5,0.0}

\makeatletter
\newcounter{HALG@line}
\renewcommand{\theHALG@line}{\thealgorithm.\arabic{ALG@line}}
\makeatother

\makeatletter
\newcommand{\algmargin}{\the\ALG@thistlm}
\makeatother
\algnewcommand{\parState}[1]{\State\parbox[t]{\dimexpr\linewidth-\algmargin}{\strut #1\strut}}

\begin{document}

\title{%
	A penalty-type method for relaxed inverse optimal control problems
	}%
\author{%
	Floriane Mefo Kue%
	\footnote{University of Siegen,
		Department of Mathematics,
		57068 Siegen,
		Germany,\\Department of Mathematics, Higher Training College, University of Yaoundé I, P.O. Box 812, Yaoundé, Cameroon,\\
AIMS-Cameroon Research Center, P.O.\ Box 608, Crystal Gardens, Limbe, Cameroon
		\email{floriane.kue@aims-cameroon.org},
		\orcid{0000-0003-0827-0809}
		}
	\and
	Patrick Mehlitz%
	\footnote{%
		University of Marburg,
		Department of Mathematics and Computer Science,
		35032 Marburg,
		Germany,
		\email{mehlitz@uni-marburg.de},
		\orcid{0000-0002-9355-850X}%
		}
	\and
	Thorsten Raasch%
	\footnote{University of Siegen,
		Department of Mathematics,
		57068 Siegen,
		Germany,
		\email{raasch@mathematik.uni-siegen.de},
		\orcid{0000-0002-2607-3115}
	}%
	}

\publishers{}
\maketitle

\begin{abstract}
	This paper is devoted to the introduction and analysis 
	of a penalty-type method for the numerical treatment of 
	a class of bilevel optimization problems arising from inverse optimal control.
	The algorithm is designed to compute stationary points of the associated relaxed value function reformulation.
	This is achieved by determining a sequence of stationary points associated with a sequence of surrogate problems
	where the relaxed value function constraint is penalized, 
	where the updates of upper- and lower-level decision variables are decoupled,
	and where the penalty parameter is enlarged only in those iterations 
	which do not come along with a sufficient improvement of some feasibility measure.
	The resulting method does not comprise any linesearch, the lower-level problem has to be evaluated just once per iteration, 
	and the penalty parameter does not need to be driven to infinity.
	Nevertheless, subsequential convergence results are obtained under reasonable assumptions.
	Numerical experiments, where the relaxation parameter is also driven to zero, visualize effectiveness of the approach.
\end{abstract}


\begin{msc}	
	\mscLink{49K20}, \mscLink{49M20}, \mscLink{49N45}, \mscLink{90C48}
\end{msc}

\section{Introduction}\label{sec:intro}

In bilevel optimization, hierarchical optimization problems with two decision makers are considered 
where the objective function and the feasible set of the so-called upper-level decision maker
(often referred to as the leader) depend on the solution set of a second,
parametric in the leader's variables, so-called lower-level optimization problem (whose decision maker
is called the follower). In the underlying decision process, first, the leader fixes his variables
which are then handed over to the follower who now is in position to solve his problem to global optimality.
The associated solution set is then passed back to the leader who, thus, can evaluate his objective function.
Often, it is assumed that the leader is allowed to choose freely from the follower's solution set,
corresponding to the so-called optimistic approach to bilevel optimization. 
Nowadays, bilevel optimization is one of the most popular research areas in operations research and
mathematical optimization as the underlying model paradigm, exemplary, covers hyperparameter tuning
in machine learning models, see \cite{GrazziPontilSalzoZemkoho2026} for a recent survey,
which can be rated as a special instance of inverse optimization where parameters in optimization problems
have to be reconstructed from given noisy observations.
A detailed introduction to bilevel optimization can be found in the monographs 
\cite{bard2013practical, dempe2002foundations, shimizu2012nondifferentiable}. 

A growing subfield of bilevel optimization is the one of so-called bilevel optimal control
where at least one the involved decision makers faces constraints comprising ordinary or partial
differential equations and, thus, has to solve an optimal control problem,
see \cite{hinze2008optimization,lewis2012optimal,troltzsch2010optimal,troutman2012variational} 
for an introduction to optimal control.
Applications of bilevel optimal control comprise 
gas balancing and transport, see \cite{GoettlichMehlitzSchillinger2024,KalashnikovBenitaMehlitz2015},
scheduling of multi-agent systems, e.g., planes and cranes,
see \cite{fisch2012solution,knauer2010hybrid},
or parameter tracking in models of human locomotion,
see \cite{hatz2012estimating}.
An overview of the topic is presented in the survey book chapter \cite{MehlitzWachsmuth2020}.
The subfield of bilevel optimal control where parameters in optimal control problems have to be reconstructed
from given noisy observations is referred to as inverse optimal control,
and some theory for this particular setting has been developed, e.g., in the recent contributions
\cite{dempe2019solving,DempeHarderMehlitzWachsmuth2022,friedemann2023finding,harder2019optimality,HollerKunischBarnard2018}.

In practice, bilevel optimization problems are generally not stable. 
For some values of the leader's variables, the lower-level problem may possess multiple global minimizers. 
Then small perturbations of the problem data can lead to dramatic changes of the lower-level solution set and, thus,
of the overall bilevel optimization problem.
Exemplary, this has been illustrated by means of several examples in \cite[Section~7.2]{dempe2002foundations}
and is mainly caused by the fact that single-level reformulations of bilevel optimization problems
suffer from an intrinsic lack of regularity, i.e., constraint qualifications of reasonable strength are
violated at each feasible point, see \cite{DempeMehlitz2025} for an overview and suitable references.
To overcome this undesirable behavior of bilevel optimization problems, 
it has been suggested in the literature to equip the leader with the set of $\varepsilon$-optimal solutions
of the lower-level problem for some fixed $\varepsilon>0$,
see \cite{loridan1988approximate, loridan1989new,loridan2006varepsilon}.
The resulting relaxed bilevel optimization problem may be interpreted as a robustifed version of the original one.
Indeed, there is a fairly good chance that standard constraint qualifications hold at the feasible points
of single-level reformulations associated with the relaxed bilevel optimization problem.
This observation also has been exploited for numerical purposes.
Indeed, in \cite{lampariello2020numerically,LamparielloSagratellaSassoStein2023,lin2014solving,ye2023difference},
solution algorithms for bilevel optimization problems based on the relaxed so-called value function reformulation have been
suggested, analyzed, and tested.
Recently, similar work has been done on the base of the relaxed Moreau envelope reformulation,
see \cite{BaiZengZhangZhang2026,GaoYeYinZengZhang2026}.

In this paper, we propose a simple penalty-type solution method for a class of inverse optimal control problems
which builds upon the associated relaxed value function reformulation.
As a model problem, we exploit the inverse optimal control problem considered in 
\cite{dempe2019solving,DempeHarderMehlitzWachsmuth2022,friedemann2023finding}.
In these papers, a global solution method has been suggested and studied which evolves from the value function reformulation
by approximating the lower-level optimal value function therein from above by an iteratively refined piecewise affine function
and decomposing the resulting relaxations into finitely many convex subproblems which can be solved efficiently with the aid
of a semismooth Newton-type method.
The approach considered here is completely different and, from the viewpoint of computation cost, much cheaper 
as we merely aim at the identification of stationary points of the relaxed value function reformulation.
The latter is a reasonable goal as Robinson's constraint qualification holds at each feasible point of the problem of interest.
We penalize the relaxed value function constraint with the aid of a nonsmooth penalty function
in order to exploit the advantage of exact penalization.
Then we iteratively solve the stationarity system of the arising penalized subproblem as the penalty parameter grows
in some but not necessarily all iterations.
In each iteration, the penalty parameter is enlarged only if a certain feasibility measure does not decrease enough,
and we distinguish between the treatment of the leader's and follower's variables
which are updated by a simple projection step and the solution of a square mixed-nonlinear complementarity system, respectively.
This leads to a simple to implement,
linesearch free method which requires merely one evaluation of the lower-level optimal control problem in each iteration
and does not require the penalty parameter to be send to infinity.
We analyze the convergence properties of the suggested algorithm and present results of computational experiments.

The remainder of the paper is organized as follows.
In \cref{sec:notation}, we introduce the notation which will be used throughout. 
\cref{sec:problem_statement} presents the precise problem statement and recalls some preliminary results on the model problem taken from \cite{dempe2019solving}. 
\cref{sec:algorithm} is dedicated to the introduction of the announced penalty-type solution algorithm for the model problem and also presents some convergence results.
Details on its implementation and results of numerical experiments can be found in \cref{sec:experiments}.
In \cref{sec:outro}, we close the paper with some final remarks and suggestions for future research. 

\section{Notation}\label{sec:notation}

In this paper, we equip $\R^n$, the space of all real vectors with $n\in\mathbb{N}$ components, 
with the Euclidean inner product $\dual{\cdot}{\cdot}$ and the Euclidean norm $\norm{\cdot}$. 
Given any two vectors $x,x'\in\R^n$, 
we exploit $x\perp x'$ in order to represent the orthogonality relation $\dual{x}{x'}=0$.
For an arbitrary Banach space $\XX$, we denote its norm by $\norm{\cdot}_{\XX}$. 
The symbol $\XX^*$ stands for the (topological) dual space of $\XX$, and
$\dual{ \cdot}{\cdot}_\XX \colon \XX^*\times \XX \to \mathbb{R}$
represents the associated dual pairing.
Convergence of a sequence $\{x_k\}_{k\in\N}\subset\XX$ to $x\in\XX$ will be denoted by $x_k\to x$,
while the weak convergence of $\{x_k\}_{k\in\N}$ to $x$ is expressed via $x_k\rightharpoonup x$.
Whenever $\XX$ is reflexive, we identify the bidual space $\XX^{**}$ with $\XX$.
If $\XX$ is a Hilbert space, we may identify $\XX^*$ with $\XX$ via the Riesz isomorphism,
but we will clearly indicated in the paper whenever we make use of this identification.
Throughout, all Banach and Hilbert spaces are assumed to be real.

Let $\XX$ be a Hilbert space, fix a nonempty, closed, convex set $A\subset \XX$, and pick $x\in\XX$.
Then $\proj_A(x)$ denotes the projection of $x$ onto $A$.
Furthermore, for $\bar x\in A$,
\[
	R_A(\bar x)
	:=
	\bigcup\limits_{s\geq 0}s(A-\{\bar x\}), \qquad
	N_A(\bar x)
	:=
	\left\{
	x^* \in \XX^* \,\middle|\,
	\forall x\in A\colon\,\dual{ x^*}{ x - \bar x }_\XX \le 0
	\right\}
\]
denote the radial and normal cone to $A$ at $\bar x$, respectively.
Whenever $\XX^*$ is identified with $\XX$ via the Riesz isomorphism, we have the equivalence
\begin{equation}\label{eq:char_proj}
	\bar x = \proj_A(x)\quad\Longleftrightarrow\quad x-\bar x\in N_A(\bar x).
\end{equation}
For the purpose of completeness, we set $R_A(\tilde x):=\emptyset$ and $N_A(\tilde x):=\emptyset$ whenever $\tilde x\notin A$.
Given sequences $\{x_k\}_{k\in\N}\subset A$ and $\{x_k^*\}_{k\in\N}\subset \XX^*$ such that $x_k^*\in N_A(x_k)$ holds for all $k\in\N$
while the convergences $x_k\to \bar x$ and $x_k^*\rightharpoonup x^*$ are valid for some $\bar x\in A$ and $x^*\in \XX^*$,
it easily follows from the definition of the normal cone that $x^*\in N_A(\bar x)$.
This property will be referred to as the closedness of the normal cone mapping $N_A(\cdot)$ throughout.

For yet another Hilbert space $\YY$, we denote by $\LL(\XX,\YY)$ the Banach space of all linear continuous operators mapping from $\XX$ to $\YY$.  
The adjoint of an operator $F \in \LL(\XX,\YY)$ is represented by $F^*\in\LL(\YY^*,\XX^*)$.
Whenever $\XX\subset\YY$ and the identity is continuous as a mapping from $\XX$ to $\YY$, we write $\XX\hookrightarrow\YY$.
For a Fr\'{e}chet differentiable function $G\colon \XX\to \YY$ and $\bar x\in \XX$, $G'(\bar x)\in\LL(\XX,\YY)$ denotes the Fr\'{e}chet derivative of $G$ at $\bar x$.
Whenever $g\colon \XX\to\R$ is Fr\'{e}chet differentiable at $\bar x\in\XX$, with a slight abuse of notation,
we identify $g'(\bar x)$ with its gradient $\nabla g(\bar x):=g'(\bar x)^*1$.
Partial Fr\'{e}chet derivatives with respect to certain blocks of variables are denoted in the canonical way.

Consider a set $D:=\{x\in \Omega\,|\,G(x)\in C\}$, where $G\colon \XX\to \YY$ is a continuously Fr\'{e}chet differentiable mapping between Hilbert spaces $\XX$ and $\YY$,
$\Omega\subset\XX$ is nonempty, closed, and convex, and $C\subset \YY$ is nonempty, closed, and convex.
We say that Robinson's constraint qualification, see \cite{Robinson1976,ZoweKurcyusz1979}, holds at $\bar x\in D$
whenever the condition
\[
	G'(\bar x) R_\Omega(\bar x) - R_C(G(\bar x)) = \YY
\]
is valid. This qualification condition is an essential tool to verify existence of Lagrange multipliers
for optimization problems stated in Banach spaces.

For a bounded domain $\Omega \subset \mathbb{R}^d$, we denote by $L^2(\Omega)$ the Lebesgue space of measurable and square-integrable functions,
and we identify $L^2(\Omega)^*$ with $L^2(\Omega)$.  
Moreover, $H_0^1(\Omega)$ represents the closure of $C_0^\infty(\Omega)$, the space of all infinitely many times differentiable functions with compact support in $\Omega$, 
with respect to the standard $H^1$-Sobolev norm. Its dual space is denoted by $H^{-1}(\Omega) := H_0^1(\Omega)^*$.
We note that $H^1_0(\Omega)\hookrightarrow L^2(\Omega)\hookrightarrow H^{-1}(\Omega)$.  

\section{Problem statement and properties}\label{sec:problem_statement}

To start, let us consider the (parametric in $x\in\R^n$) abstract 
optimal control problem
\begin{equation}\label{Lower_level_problem}\tag{LL$(x)$}
	\begin{aligned}
 		&\min\limits_{y,u}&  &\frac{1}{2}\|y-Ex\|^2_{\VV}+\frac{\sigma}{2}\|u\|^2_{\UU}& \\
 		&\stt& &A y=B u,\,u \in U_{\ad},&
   \end{aligned}
\end{equation} 
whose variables are the state $y$ and the control $u$ 
which have to be chosen from suitable function spaces $\YY$ and $\UU$, respectively,
such that the state equation $Ay=Bu$ and certain control constraints $u\in U_\ad$
hold. Typically, $U_\ad$ is a box induced by upper and lower bounds on the control.
The objective function in \eqref{Lower_level_problem} enforces the state to
be, in an observation space $\VV$ such that $\YY\hookrightarrow\VV$, 
as close as possible to a desired state $Ex$, where $E$ is a linear continuous operator,
while the control effort is reasonably small.

We aim to identify the  true value of the desired state $Ex$ from given (potentially noisy) measurements
in a superordinate optimization problem
\begin{equation}\label{eq:IOC}
	\begin{aligned}
		&\min\limits_{x,y,u}& 	& F(x, y, u)& \\
 		&\stt& 					& x \in X_{\ad},\, (y,u) \text{ solves \eqref{Lower_level_problem},} &
	\end{aligned}
\end{equation}
which is, thus, an inverse optimal control problem.
Here, the function $F\colon\R^n\times\YY\times\UU\to\R$ represents a measure 
which compares the true solutions of \eqref{Lower_level_problem} with some observed data,
e.g.,
\[
	F(x,y,u) := \frac12\norm{y-y_\textup{o}}_{\VV}^2 + \frac12\norm{u-u_\textup{o}}_{\UU}^2
\]
for given $(y_\textup{o},u_\textup{o})\in\YY\times\UU$,
but we emphasize that other choices for $F$ are also covered by the analysis in this paper.
Candidate parameters are taken from the set $X_\ad$.
A closely related model problem has been investigated in \cite{dempe2019solving,DempeHarderMehlitzWachsmuth2022,friedemann2023finding}.
Recall that, in the terminology of bilevel optimization,
the inner problem \eqref{Lower_level_problem} is referred to as the lower-level problem
while the outer problem \eqref{eq:IOC} is called the upper-level problem.

Let $\varphi\colon\R^n\to\R\cup\{-\infty,\infty\}$ be the optimal value function associated with \eqref{Lower_level_problem},
i.e., it assigns to each $x\in\R^n$ the optimal objective function value of \eqref{Lower_level_problem}.
Then \eqref{eq:IOC} is, obviously, equivalent to the single-level problem
\begin{equation}\label{eq:IOC_ref}\tag{IOC}
	\begin{aligned}
 		&\min\limits_{x,y,u}& 	& F(x, y, u)& \\
 		&\stt& 		& x \in X_{\ad},\,A y=B u,\,u \in U_{\ad},& \\
        &&          & f(x, y, u)-\varphi(x) \leq 0,&
   \end{aligned}
\end{equation}
where we used $f\colon\R^n\times\YY\times\UU\to\R$ in order to abbreviate the objective function of the 
lower-level problem \eqref{Lower_level_problem} given by
\begin{equation}\label{eq:def_f}
	f(x,y,u) := \frac12\norm{y-Ex}_{\VV}^2 + \frac{\sigma}{2}\norm{u}_{\UU}^2.
\end{equation}
Despite being equivalent to \eqref{eq:IOC}, \eqref{eq:IOC_ref} comes along with several disadvantages.
First and foremost, the function $\varphi$ is only implicitly known,
and evaluation of its function values and (generalized) derivatives might be costly.
Second, as we will recall below, Robinson's constraint qualification fails to hold at all feasible points of \eqref{eq:IOC_ref}.

In order to overcome this drawback, we consider 
the relaxed inverse optimal control problem
\begin{equation}\label{main_problem}\tag{IOC$_\varepsilon$}
	\begin{aligned}
 		&\min\limits_{x,y,u}& 	& F(x, y, u)& \\
 		&\stt& 		& x \in X_{\ad},\,A y=B u,\,u \in U_{\ad},& \\
        &&                 		& f(x, y, u)-\varphi(x) \leq \varepsilon,&
   \end{aligned}
\end{equation}
where $\varepsilon > 0$ is a relaxation parameter
(we note that choosing $\varepsilon:=0$ in \eqref{main_problem},
which we omit here, recovers problem \eqref{eq:IOC_ref}).
Indeed, in contrast to \eqref{eq:IOC_ref}, \eqref{main_problem} is a rather regular problem
as Robinson's constraint qualification holds at each feasible point.

Throughout the paper, we will make use of the assumptions summarized below.

\begin{assumption}\label{ass:setting}
	Let $\YY$, $\UU$, and $\VV$ be Hilbert spaces such that $\YY\hookrightarrow\VV$,
	and $\UU^*$ is identified with $\UU$ while $\VV^*$ is identified with $\VV$.
	We assume that
	$F\colon\mathbb{R}^n \times \YY \times \UU \to \mathbb{R}$
	is continuously Fr\'{e}chet differentiable and convex. 
	The set $X_{\ad} \subset \mathbb{R}^n$ is nonempty, convex, and compact.
	Furthermore, $A \in \mathcal{L}(\YY,\YY^*)$ is an isomorphism,
	and $B \in \mathcal{L}(\UU,\YY^*)$ as well as $E \in \mathcal{L}(\R^n,\VV)$ are two other linear continuous operators. 
	The admissible control set $U_{\ad} \subset \UU$ is assumed to be nonempty, closed,
 	convex, and bounded. Furthermore, $\sigma > 0$ is a regularization parameter.
 	The function $f$ is given as in \eqref{eq:def_f}.
 	Finally, we fix $\varepsilon>0$.
\end{assumption}

From now on \cref{ass:setting} will be treated as standing without further notice.
 
Throughout the paper, let $S:=A^{-1}B$ denote the (linear continuous) solution operator 
of the linear state equation $Ay=Bu$. Then, for each $x\in\R^n$, $y=Su$ holds for each feasible pair $(y,u)\in\YY\times\UU$
of \eqref{Lower_level_problem}.
Furthermore, we suppress the linear continuous operators 
representing the natural embedding $\YY\hookrightarrow\VV$ and the associated adjoint embedding $\VV\hookrightarrow\YY^*$
for simplicity of notation.

We also note that the lower-level objective function $f$ defined in \eqref{eq:def_f} is continuously Fr\'{e}chet differentiable and convex with partial Fr\'{e}chet derivatives
\begin{align*}
	f'_x(x,y,u)
	=
	E^*(Ex-y),
	\qquad
	f'_y(x,y,u)
	=
	y - Ex,
	\qquad
	f'_u(x,y,u)
	=
	\sigma\,u.
\end{align*}

In the upcoming subsections, 
we collect some results which have been obtained in \cite{dempe2019solving} already
and will become handy later on.

\subsection{The lower-level problem}

To start, let us recall some properties of the lower-level problem  \eqref{Lower_level_problem}.
Our first proposition, which recapitulates \cite[Lemma~4.1]{dempe2019solving}, comments on the
solution behavior of \eqref{Lower_level_problem}.

\begin{proposition}\label{Prop_llp_1}
 	For each $x\in\R^n$, problem \eqref{Lower_level_problem} is well-defined
 	and admits a uniquely determined global minimizer $(y(x),u(x))\in\YY\times\UU$.
 	Moreover, the solution mappings
	\[
 		\R^n\ni x\mapsto y(x)\in\YY,
 		\qquad
  		\R^n\ni x\mapsto u(x)\in\UU
	\]
	are Lipschitz continuous.
\end{proposition}

Let us note that the proof of \cite[Lemma~4.1]{dempe2019solving} allows to find
explicit Lipschitz constants of $x\mapsto y(x)$ and $x\mapsto u(x)$ in terms of
initial problem data.

\cref{Prop_llp_1} implies that the optimal value function $\varphi$ of \eqref{Lower_level_problem}
enjoys the representation
\[
	\varphi(x) 
	= 
	f(x,y(x),u(x)) 
	= 
	\frac{1}{2}\norm{y(x) - Ex}^2_{\VV} + \frac\sigma2\norm{u(x)}^2_{\UU}
\]
for all $x\in\R^n$
and is real-valued and locally Lipschitz continuous.
Here, we are also concerned with differentiability properties of $\varphi$.
For this purpose, let us recall \cite[Lemma~4.3]{dempe2019solving}.

\begin{proposition}\label{prop:smooth_value_function} 
	The optimal value function $\varphi$ of \eqref{Lower_level_problem} is continuously Fr\'{e}chet differentiable and convex. 
	Given $x \in \mathbb{R}^n$, 
	the associated Fr\'{e}chet derivative is given by
	\[
		\varphi'(x) = f'_x(x,y(x),u(x)) = E^*(Ex - y(x)),
	\]
	where $(y(x),u(x))\in\YY\times\UU$ is the uniquely determined global minimizer of \eqref{Lower_level_problem},
	see \cref{Prop_llp_1}.
\end{proposition}

\subsection{The upper-level problem}

First, let us mention that \eqref{eq:IOC_ref} and \eqref{main_problem} possess global minimizers,
see \cite[Corollary~4.1, Lemma~5.3]{dempe2019solving}. 
\begin{proposition}\label{prop:existence} 
	\hfill
	\begin{enumerate}
	\item Problem \eqref{eq:IOC_ref} possesses a global minimizer
	\item Problem \eqref{main_problem} possesses a global minimizer.
	\end{enumerate}
\end{proposition}

We also would like to note that, according to \cite[Theorem~5.1]{dempe2019solving}, taking the limit in a family $(x_\varepsilon,y_\varepsilon,u_\varepsilon)$
of global minimizers associated with \eqref{main_problem} as $\varepsilon$ tends to zero recovers global minimizers of \eqref{eq:IOC_ref}
and, thus, \eqref{eq:IOC}. In this regard, we may refer to \eqref{main_problem} as a robustification of \eqref{eq:IOC_ref}
as already mentioned in \cref{sec:intro}.

In \cite[Lemmas~5.1 and~5.2]{dempe2019solving}, regularity of problems \eqref{eq:IOC_ref} and \eqref{main_problem} has been studied.
Below, we summarize these results which are essential for our further investigations.

\begin{proposition}\label{prop:RCQ}
	\hfill
	\begin{enumerate}
		\item Robinson's constraint qualification is violated at each feasible point of \eqref{eq:IOC_ref}.
		\item Robinson's constraint qualification is valid at each feasible point of \eqref{main_problem}.
	\end{enumerate}
\end{proposition}

\section{A penalty-type solution algorithm}\label{sec:algorithm}

In this section, 
we develop an algorithm for the computational solution of problem \eqref{main_problem}. 
The underlying penalty approach is motivated in \cref{sec:penalty_approach}.
A special parametric penalized surrogate problem is investigated in \cref{sec:penalized_subproblem}.
In \cref{sec:algo}, we state and discuss the suggested method,
and \cref{sec:ana} presents the associated convergence analysis.

\subsection{A penalization approach}\label{sec:penalty_approach}

To start the subsection,
let us elaborate on the stationarity conditions of \eqref{main_problem}.

\begin{proposition}\label{prop:NOC_IOCeps}
 	Let $(\bar x,\bar y,\bar u)  \in \R^n\times\YY\times \UU $ be a local minimizer of \eqref{main_problem}.
 	Then there exist multipliers $(\lambda^x,\lambda^p,\lambda^u,\lambda^\textup{vf})\in\R^n\times\YY\times\UU\times\R$
	which satisfy the following conditions:
	\begin{subequations}
		\label{eq:optimality-system}
		\begin{align}
			\label{eq:optimality_system_x}
				0
				&=
				F'_x(\bar{x}, \bar{y},  \bar u) + \lambda^x+\lambda^\textup{vf} \left[ f_x'(\bar{x},\bar{y},\bar{u}) - E^{*}\left( E\bar x - y(\bar x)\right) \right],
				\\
			\label{eq:optimality_system_y}
				0
				&=
				F'_y(\bar{x}, \bar{y},  \bar u) + A^*\lambda^p+\lambda^\textup{vf}f'_y(\bar x,\bar y,\bar u)
				,
				\\
			\label{eq:optimality_system_u}
				0
				&=
				F'_u(\bar{x}, \bar{y},  \bar u) - B^*\lambda^p + \lambda^u + \lambda^\textup{vf}f'_u(\bar x,\bar y,\bar u)
				,
				\\
			\label{eq:optimality_system_Xad}
				\lambda^x&\in N_{X_\ad}(\bar x),
				\\
			\label{eq:optimality_system_Uad}
				\lambda^u &\in N_{U_\ad}(\bar u),
				\\
			\label{eq:optimality_system_vf}
				0 &\leq \lambda^\textup{vf} \perp f(\bar x,\bar y,\bar u)-\varphi(\bar x)-\varepsilon \leq 0.
	\end{align} 
	\end{subequations}
\end{proposition}
\begin{proof}
	According to \cref{prop:smooth_value_function,prop:RCQ}, 
	\eqref{main_problem} is a smooth optimization problem satisfying Robinson's constraint qualification
	at $(\bar x,\bar y,\bar u)$.
	As the latter is assumed to be a local minimizer of \eqref{main_problem},
	it is a stationary point,
	and standard calculations show that \eqref{eq:optimality-system}
	is the stationarity system of \eqref{main_problem},
	noting that $\varphi'(\bar x) = E^*(E\bar x - y(\bar x))$ holds by \cref{prop:smooth_value_function}.
\end{proof}

Let us point the reader to the fact that condition \eqref{eq:optimality_system_vf} is rather challenging to treat.
Apart from the fact that it is a complementarity condition,
it involves the function $\varphi$ which is not explicitly known.
Applying standard solvers for mixed-nonlinear complementarity systems to solve \eqref{eq:optimality-system}
is, thus, not recommendable as these typically require many evaluations of $\varphi$,
e.g., for linesearch procedures in (globalized) nonsmooth Newton-type methods. 

To overcome this computational drawback, we consider an optimization problem in which the constraint
\[ 
	f(x, y, u)-\varphi(x)-\varepsilon\leq 0
\] 
is incorporated directly into the objective function
via a nonsmooth penalty term.
More precisely, for a penalty parameter $\alpha>0$, we are concerned with
\begin{equation}\label{penalized_problem}\tag{IOC$_{\varepsilon}(\alpha)$}
	\begin{aligned}
 	&\min\limits_{x,y,u}& 	& F(x, y, u) + \alpha\max\{f(x, y, u)-\varphi(x)-\varepsilon,0\}& \\
 	&\mbox{ s.t. }& 		& x \in X_{\ad},\, A y=B u,\,u \in U_{\ad}.& 
   	\end{aligned}
\end{equation}
Let us elaborate on suitable stationarity conditions for this problem,
where the nonsmoothness in the objective function is treated in terms of Clarke's subdifferential construction, see \cite{clarke1990optimization}.
Throughout, we will denote the associated subdifferentiation operator by $\partial$.
Its precise definition is not required here, we will merely rely on its well-established calculus.

Based on the elementary sum rule from \cite[Proposition~2.3.3]{clarke1990optimization}, 
a feasible point $(\bar x,\bar y,\bar u)\in\R^n\times\YY\times\UU$
of \eqref{penalized_problem} is stationary 
if and only if there exist multipliers 
$(\lambda^x,\lambda^p,\lambda^u)\in\R^n\times\YY\times\UU$ such that
\begin{align*}
	-F'(\bar x,\bar y,\bar u) - (\lambda^x,A^*\lambda^p,-B^*\lambda^p+\lambda^u) 
	\in 
	\alpha \partial \chi(\bar x,\bar y,\bar u)
\end{align*}
as well as \eqref{eq:optimality_system_Xad} and \eqref{eq:optimality_system_Uad} hold,
where $\chi\colon\R^n\times\YY\times\UU\to\R$ is the locally Lipschitz continuous function 
given by
\[
	\chi(x,y,u) 
	:=
	\max\{f(x,y,u)-\varphi(x)-\varepsilon,0\}.
\]
Based on the maximum rule from \cite[Proposition~2.3.12]{clarke1990optimization}, we find
\[
	\partial\chi(\bar x,\bar y,\bar u)
	=
	\begin{cases}
		\{t\,\theta'(\bar x,\bar y,\bar u)\,|\,t\in[0,1]\}	&	\theta(\bar x,\bar y,\bar u) = 0,
		\\
		\{\theta'(\bar x,\bar y,\bar u)\}					&	\theta(\bar x,\bar y,\bar u) > 0,
		\\
		\{(0,0,0)\}											&	\theta(\bar x,\bar y,\bar u) < 0,
	\end{cases}
\]
where we used the continuously Fr\'{e}chet differentiable function $\theta\colon\R^n\times\YY\times\UU\to\R$ defined via
\[
	\theta(x,y,u)
	:=
	f(x,y,u) - \varphi(x) - \varepsilon
\]
whose derivative at $(\bar x,\bar y,\bar u)$ is given by
\[
	\theta'(\bar x,\bar y,\bar u)
	=
	f'(\bar x,\bar y,\bar u) - (E^*(E\bar x - y(\bar x)),0,0),
\]
see \cref{prop:smooth_value_function}.
Putting everything together, 
we obtain the stationarity conditions 
\begin{subequations}\label{eq:optimality-system_1}
	\begin{align}
		\label{eq:optimality-system_1x}
			0
			&=
			F'_x(\bar{x}, \bar{y},  \bar u) + \lambda^x+\alpha t\left[ f_x'(\bar{x},\bar{y},\bar{u}) 
			- 
			E^{*}\left( E\bar x - y(\bar x)\right) \right]
			,
			\\
		\label{eq:optimality_system_1y}
			0
			&=
			F'_y(\bar{x}, \bar{y},  \bar u) + A^* \lambda^p+\alpha t f'_y(\bar x,\bar y,\bar u)
			,
			\\
		\label{eq:optimality_system_1u_1}
			0
			&=
			F'_u(\bar{x}, \bar{y},  \bar u) -B^*\lambda^p + \lambda^u +\alpha tf'_u(\bar x,\bar y,\bar u)
			,
			\\
		\label{eq:optimality_system_1vf_1}
			t&\in \partial\max\{\cdot,0\}(f(\bar x,\bar y,\bar u)-\varphi(\bar x)-\varepsilon),
	\end{align} 
\end{subequations}
as well as \eqref{eq:optimality_system_Xad} and \eqref{eq:optimality_system_Uad},
where we used
\[
	\partial\max\{\cdot,0\}(s)
	=
	\begin{cases}
		[0,1]	&	s=0,\\
		\{1\}	&	s>0,\\
		\{0\}	&	s<0.
	\end{cases}
\]

The following result is decisive for our subsequent investigations.

\begin{proposition}\label{prop:NOC_IOPeps_pen}
 	Let $(\bar{x}, \bar{y}, \bar{u})\in\R^n\times\YY\times\UU$ 
  	be a stationary point of \eqref{main_problem},
  	and let 
  	$(\lambda^x,\lambda^p,\lambda^u,\lambda^\textup{vf})\in\R^n\times\YY\times\UU\times\R$ 
  	be associated multipliers solving the stationarity system \eqref{eq:optimality-system}.
  	Then, for each $\alpha \geq \lambda^\textup{vf}$, 
  	$(\bar{x}, \bar{y}, \bar{u})$ is a stationary point 
  	of the partially penalized problem \eqref{penalized_problem}.
\end{proposition}
\begin{proof}
	We reuse the already known multipliers $(\lambda^x,\lambda^p,\lambda^u,\lambda^\textup{vf})$
	satisfying the conditions stated in \eqref{eq:optimality-system}.
	Then \eqref{eq:optimality_system_Xad} and \eqref{eq:optimality_system_Uad} hold trivially,
	and it remains to verify the conditions from \eqref{eq:optimality-system_1}.
	We note that $f(\bar x,\bar y,\bar u)-\varphi(\bar x)-\varepsilon\leq 0$ is valid
	since $(\bar x,\bar y,\bar u)$ is feasible to \eqref{main_problem}.
	If $\lambda^\textup{vf}=0$ is true, 
	we may choose $t:=0$ in order to satisfy the conditions \eqref{eq:optimality-system_1}.
	Otherwise, $\lambda^\textup{vf}>0$ holds, 
	which is only possible if $f(\bar x,\bar y,\bar u)-\varphi(\bar x)-\varepsilon = 0$ is true,
	so we can exploit $\alpha\geq\lambda^\textup{vf}$ and set $t:=\lambda^\textup{vf}/\alpha\in [0,1]$ 
	in order to see that the conditions \eqref{eq:optimality-system_1} are valid.
\end{proof}

Let us close this subsection with some brief remarks addressing the observations from \cref{prop:NOC_IOPeps_pen}.

\begin{remark}\label{rem:relation_stationarity}
	\hfill
	\begin{enumerate}
		\item We note that the assumptions of \cref{prop:NOC_IOPeps_pen} imply that $f(\bar x,\bar y,\bar u)-\varphi(\bar x)-\varepsilon\leq 0$ holds.
			Then \eqref{eq:optimality_system_1vf_1} reduces to the orthogonality-type condition
			\[
				t\in[0,1],\quad t\perp f(\bar x,\bar y,\bar u)-\varphi(\bar x)-\varepsilon.
			\]
		\item Given $\alpha>0$, let $(\bar x,\bar y,\bar u)\in \R^n\times\YY\times\UU$ be a stationary point of \eqref{penalized_problem},
			and let $(\lambda^x,\lambda^p,\lambda^u,t)\in\R^n\times\YY\times\UU\times\R$ provide a solution
			of \eqref{eq:optimality-system_1} such that \eqref{eq:optimality_system_Xad} and \eqref{eq:optimality_system_Uad} hold,
			i.e., the multipliers $(\lambda^x,\lambda^p,\lambda^u,t)$ solve the stationarity system associated with \eqref{penalized_problem}. 
			If $f(\bar x,\bar y,\bar u)-\varphi(\bar x)-\varepsilon\leq 0$ is valid,
			we observe that $(\lambda^x,\lambda^p,\lambda^u,\alpha t)$ provides a solution of \eqref{eq:optimality-system},
			i.e., $(\bar x,\bar y,\bar u)$ is a stationary point of \eqref{main_problem}.
		\item Given $\alpha>0$, let $(\bar x,\bar y,\bar u)\in \R^n\times\YY\times\UU$ be a stationary point of \eqref{penalized_problem},
			and assume that $f(\bar x,\bar y,\bar u)-\varphi(\bar x)-\varepsilon > 0$ is valid.
			Then \eqref{eq:optimality_system_1vf_1} reduces to $t=1$.
	\end{enumerate}
\end{remark}

\subsection{Analysis of a parametric penalized problem}\label{sec:penalized_subproblem}

The solution method we are going to construct will iteratively solve the stationary system associated with \eqref{penalized_problem}
while treating $\alpha$ in an adaptive way.
In each iteration of the algorithm,
we decouple the solution of the stationarity system \eqref{eq:optimality-system_1}, \eqref{eq:optimality_system_Xad}, and \eqref{eq:optimality_system_Uad}
for fixed $\alpha>0$ into two steps. First, we keep the upper-level variables fixed and merely focus on the update of the
lower-level variables and multipliers by solving a mixed-nonlinear complementarity system. In a second step, the upper-level variables will be updated
by a simple projection step.

In order to better understand the first of these two steps,
this subsection is dedicated to the analysis of the penalized parametric optimization problem
\begin{equation}\label{parametric_penalized_problem}\tag{PP$_\alpha(x)$}
\begin{aligned}
\min_{y,u}\quad &
F(x,y,u)
+\alpha\left(
f(x,y,u)
-\varphi(x)-\varepsilon
\right),\\
\text{s.t.}\quad
& Ay=Bu,\,u\in U_{\ad},
\end{aligned}
\end{equation}
where $x\in\R^n$ is the parameter and $\alpha>0$ is a fixed constant.
Let us note that \eqref{parametric_penalized_problem} is a convex optimization problem.
It is easy to see that Robinson's constraint qualification holds at each feasible point of \eqref{parametric_penalized_problem},
so that, given a global minimizer $(\bar y,\bar u)\in\YY\times\UU$ of \eqref{parametric_penalized_problem},
there exist multipliers $(\lambda^p,\lambda^u)\in\YY\times\UU$ such that the conditions
\begin{align*}
	0
	&=
	F'_y(x,\bar y,\bar u) + A^*\lambda^p + \alpha f'_y(x,\bar y,\bar u),
	\\
	0
	&=
	F'_u(x,\bar y,\bar u) - B^*\lambda^p + \lambda^u + \alpha f'_u(x,\bar y,\bar u),
\end{align*}
and \eqref{eq:optimality_system_Uad} hold. 
Using the fact that $A$ is an isomorphism,
it is clear that the multipliers $\lambda^p$ and $\lambda^u$, in the case of existence,
are uniquely determined.
Conversely, due to convexity, these stationarity conditions are also sufficient for global minimality
when combined with the feasibility conditions.

To start, let us discuss the primal-dual solution behavior of \eqref{parametric_penalized_problem}.
\begin{proposition}\label{prop_parametric_penalized_1}
	Fix $\alpha>0$.
	Assume that the partial Fr\'{e}chet derivatives $F'_y$ and $F'_u$ are Lipschitz continuous.
	Then, for each $x\in\R^n$, \eqref{parametric_penalized_problem} admits a uniquely determined solution $(y_\alpha(x),u_\alpha(x))\in\YY\times\UU$,
	and there exist uniquely determined multipliers $(\lambda^p_\alpha(x),\lambda^u_\alpha(x))\in\YY\times\UU$ solving the associated stationarity
	system. Furthermore, the mappings $y_\alpha\colon\R^n\to\YY$, $u_\alpha\colon\R^n\to\UU$, $\lambda^p_\alpha\colon\R^n\to\YY$, 
	and $\lambda^u_\alpha\colon\R^n\to\UU$ are Lipschitz continuous.
\end{proposition}
\begin{proof}
	To start, let us define the state-reduced objective function $H\colon\R^n\times\UU\to\R$ of \eqref{parametric_penalized_problem} by means of
	\[
	H(x,u)
	:=
	F(x,Su,u)
	+\alpha\left(
	f(x,Su,u)
	-\varphi(x)-\varepsilon
	\right).
	\]
	Since \(F\) and $f$ are convex while the operator \(S\) is linear, the mapping $H(x,\cdot)$ is convex.
	The presence of the quadratic regularization term in $f$, see \eqref{eq:def_f}, 
	even ensures that $H(x,\cdot)$ is $\alpha\sigma$-uniformly convex. 
	Since $U_{\ad}$ is nonempty, closed, and convex, the state-reduced problem
	\begin{equation}\label{eq:reduced_subproblem}
		\min\limits_u\quad H(x,u)\quad\text{s.t.}\quad u\in U_{\ad}
	\end{equation}
	admits a unique global minimizer, denoted by $u_\alpha(x)$.
	Hence, there is $y_\alpha(x)\in\YY$, namely $y_\alpha(x):=Su_\alpha(x)$, such that $(y_\alpha(x),u_\alpha(x))$ 
	is the uniquely determined global minimizer of \eqref{parametric_penalized_problem}.
	Furthermore, according to our earlier comments, there exist uniquely determined multipliers
	$(\lambda^p_\alpha(x),\lambda^u_\alpha(x))\in\YY\times\UU$ such that
	\begin{align*}
		0
		&=
		F'_y(x,y_\alpha(x),u_\alpha(x)) + A^*\lambda^p_\alpha(x) + \alpha f'_y(x,y_\alpha(x),u_\alpha(x)),
		\\
		0
		&=
		F'_u(x,y_\alpha(x),u_\alpha(x)) - B^*\lambda^p_\alpha(x) + \lambda^u_\alpha(x) 
		+ 
		\alpha f'_u(x,y_\alpha(x),u_\alpha(x)),
		\\
		\lambda^u_\alpha(x)&\in N_{U_\ad}(u_\alpha(x)).
	\end{align*}

	Next, we prove Lipschitz continuity of the primal-dual solution mappings $y_\alpha$,
	$u_\alpha$, $\lambda^p_\alpha$, and $\lambda^u_\alpha$.
	Let $x_1,x_2\in \R^n$ be chosen arbitrarily, and set $y_i:=y_\alpha(x_i)$ and $u_i:=u_\alpha(x_i)$
	as well as $\lambda^p_i:=\lambda^p_\alpha(x_i)$ and $\lambda^u_i:=\lambda^u_i(x_i)$ for each $i=1,2$. 
	Hence, we know
	\begin{subequations}\label{eq:PP_stat}
	\begin{align}
		\label{eq:PP_stat_y}
		0
		&=
		F'_y(x_i,y_i,u_i) + A^*\lambda^p_i + \alpha (y_i - Ex_i),
		\\
		\label{eq:PP_stat_u}
		0
		&=
		F'_u(x_i,y_i,u_i) - B^*\lambda^p_i + \lambda^u_i + \alpha \sigma u_i,
		\\
		\label{eq:PP_stat_lam_u}
		\lambda^u_i&\in N_{U_\ad}(u_i)
	\end{align}
	\end{subequations}
	for $i=1,2$, where we used the precise definition of $f$ from \eqref{eq:def_f}.
	Furthermore, we note the relation
	\[
		-H'_u(x_i,u_i) = \lambda^u_i
	\]
	for $i=1,2$ which results from
	the optimality conditions associated with the state-reduced problem \eqref{eq:reduced_subproblem}.	
	Exploiting the latter, \eqref{eq:PP_stat_lam_u}, and the definition of the normal cone, we find
	\[
		\dual{H'_u(x_1,u_1)}{u_2-u_1}_{\UU} \geq 0,
		\quad
		\dual{H'_u(x_2,u_2)}{u_1-u_2}_{\UU} \geq 0,
	\]
	and adding up these inequalities yields
	\[
		\dual{H'_u(x_1,u_1) - H'_u(x_2,u_2)}{u_1-u_2}_{\UU}\leq 0.
	\]
	Next, we add $\dual{H'_u(x_2,u_1)}{u_1-u_2}_{\UU}$ on both sides
	so that some rearrangements give
	\[
		\dual{H'_u(x_2,u_1) - H'_u(x_2,u_2)}{u_1-u_2}_{\UU}
		\leq
		-\dual{H'_u(x_1,u_1) - H'_u(x_2,u_1)}{u_1-u_2}_{\UU}.
	\]
	Since $H(x_2,\cdot)$ is $\alpha\sigma$-uniformly convex, we find
	\[
		\alpha\sigma
		\norm{u_1-u_2}_{\UU}^2
		\leq
		-\dual{H'_u(x_1,u_1) - H'_u(x_2,u_1)}{u_1-u_2}_{\UU},
	\]
	which yields
	\[
		\alpha\sigma\norm{u_1-u_2}_{\UU}
		\leq
		\norm{H'_u(x_1,u_1) - H'_u(x_2,u_1)}_{\UU}.
	\]
	Noting that 
	\[
		H'_u(x,u)
		=
		S^*F_y'(x,Su,u)
		+
		F_u'(x,Su,u)
		+
		\alpha S^*(Su-Ex)
		+
		\alpha\sigma u
	\]
	holds for all $(x,u)\in\R^n\times\UU$, we find
	\begin{align*}
		\norm{H'_u(x_1,u_1) - H'_u(x_2,u_1)}_{\UU}
		\leq 
		(c_1L_{F'_y} + L_{F'_u} + \alpha c_2)\norm{x_1-x_2},
	\end{align*}
	where $L_{F'_y}\geq 0$ and $L_{F'_u}\geq 0$ denote the Lipschitz constants
	of $F'_y$ and $F'_u$, respectively, and 
	where we used $c_1:=\norm{S^*}_{\LL(\YY^*,\UU)}$ and $c_2:=\norm{S^*E}_{\LL(\R^n,\UU)}$.
	Thus, we end up with
	\begin{equation}\label{eq:explicit_Lipschitz_estimate}
		\norm{u_1-u_2}_{\UU}
		\leq
		\frac{1}{\alpha\sigma}
		(c_1L_{F'_y} + L_{F'_u} + \alpha c_2)\norm{x_1-x_2},
	\end{equation}
	verifying Lipschitzness of $u_\alpha$.
	Recalling that $S$ is linear and continuous,
	this immediately yields the Lipschitzness of $y_\alpha$. 
	Using \eqref{eq:PP_stat_y}, the Lipschitzness of $F'_y$, and the fact that $A^*$ is a continuous isomorphism,
	Lipschitzness of $\lambda^p_\alpha$ follows.
	Finally, exploiting \eqref{eq:PP_stat_u} and the Lipschitzness of $F'_u$,
	we also obtain Lipschitzness of $\lambda^u_\alpha$.
\end{proof}

Using \eqref{eq:explicit_Lipschitz_estimate} 
and the last part of the proof of \cref{prop_parametric_penalized_1},
it is possible to find explicit Lipschitz constants for the mappings
$y_\alpha$, $u_\alpha$, $\lambda^p_\alpha$, and $\lambda^u_\alpha$.
Whenever a lower bound $\alpha_0>0$ for $\alpha$ is available,
we also note that \eqref{eq:explicit_Lipschitz_estimate} allows to find
a uniform Lipschitz constant for $y_\alpha$ and $u_\alpha$ for all $\alpha\geq\alpha_0$.

In our next result, we review properties of the optimal value function $\phi_\alpha\colon\R^n\to\R$ of
\eqref{parametric_penalized_problem}, which assigns to each $x\in\R^n$ the optimal value
\begin{equation}\label{eq:yet_another_value_function}
	\phi_\alpha(x)
	:=
	F(x,y_\alpha(x),u_\alpha(x))
	+
	\alpha(f(x,y_\alpha(x),u_\alpha(x)) - \varphi(x) - \varepsilon).
\end{equation}
As we will see, the latter is continuously Fr\'{e}chet differentiable thanks to 
\cref{prop:smooth_value_function,prop_parametric_penalized_1}.

\begin{proposition}\label{prop_parametric_penalized_2}
	Fix $\alpha>0$. Assume that the partial Fr\'{e}chet derivatives $F'_y$ and $F'_u$ 
	are Lipschitz continuous.
	Then the optimal value function $\phi_\alpha\colon\R^n\to\R$ defined in \eqref{eq:yet_another_value_function} 
	is continuously Fr\'{e}chet differentiable.
	Furthermore, for each $\bar x\in\R^n$, we have
	\[
		\phi'_\alpha(\bar x)
		=
		F'_x(\bar x,y_\alpha(\bar x),u_\alpha(\bar x))+\alpha E^*(y(\bar x)-y_\alpha(\bar x)),
	\]
	where $y(\bar x)$ is the optimal state of the lower-level problem \hyperref[Lower_level_problem]{\textup{(LL$(\bar x)$)}},
	see \cref{Prop_llp_1}.
	If $F'_x$ is Lipschitz continuous, then $\phi'_\alpha$ is Lipschitz continuous as well.
\end{proposition}
\begin{proof}
	Let $J\colon\R^n\times\YY\times\UU\to\R$ be the objective function of \eqref{parametric_penalized_problem},
	and observe that the latter is continuously Fr\'{e}chet differentiable according to \cref{prop:smooth_value_function}.
	For $x,\bar x\in\R^n$, set $z:=(x,y,u)$ and $\bar z:=(\bar x,\bar y,\bar u)$,
	where $y:=y_\alpha(x)$, $u:=u_\alpha(x)$, $\bar y:=y_\alpha(\bar x)$, and $\bar u:=u_\alpha(\bar x)$
	are used for brevity of notation.
	Then, for $\ZZ:=\R^n\times\YY\times\UU$, we find
	\begin{equation}\label{prob_cont_diff_1}
	\begin{aligned}
		\phi_\alpha(x)-\phi_\alpha(\bar x)
		&=
		J(z)-J(\bar z)
		=
		\dual{J'(\bar z)}{z-\bar z}_{\ZZ} + \oo(\norm{z-\bar z}_{\ZZ}).
	\end{aligned}
	\end{equation}
	We also note that, due to \cref{prop_parametric_penalized_1}, there is a constant $c>0$ such that
	\begin{align*}
		\frac{\abs{\oo(\norm{z-\bar z}_{\ZZ})}}{\norm{z-\bar z}_{\ZZ}}
		\leq
		\frac{\abs{\oo(\norm{z-\bar z}_{\ZZ})}}{\norm{x-\bar x}}
		=
		\frac{\abs{\oo(\norm{z-\bar z})}}{\norm{z-\bar z}_{\ZZ}}
		\frac{\norm{z-\bar z}_{\ZZ}}{\norm{x-\bar x}}
		\leq
		c\frac{\abs{\oo(\norm{z-\bar z})}}{\norm{z-\bar z}_{\ZZ}}.
	\end{align*}
	Hence, we find
	\begin{equation}\label{eq:modified_oo}
		\frac{\oo(\norm{z-\bar z}_{\ZZ})}{\norm{x-\bar x}} \to 0 \quad \text{as}\quad x\to\bar x
	\end{equation}
	from the continuity of $y_\alpha$ and $u_\alpha$.
	According to \cref{prop_parametric_penalized_1} and the definition of the normal cone, 
	using $\lambda^p:=\lambda^p_\alpha(x)$ and $\bar \lambda^p:=\lambda^p_\alpha(\bar x)$,
	we find
	\[
		\begin{aligned}
			J'_y(z) + A^*\lambda^p &=0,&\quad \dual{J'_u(z)-B^*\lambda^p}{\bar u-u}_{\UU}&\geq 0,&
			\\
			J'_y(\bar z) + A^*\bar\lambda^p &=0,&\quad\dual{J'_u(\bar z)-B^*\bar\lambda^p}{u-\bar u}_{\UU}&\geq 0.&
		\end{aligned}
	\]
	Exploiting $Ay=Bu$ and $A\bar y=B\bar u$, we find $y-\bar y = S(u-\bar u)$ by linearity of the solution operator $S$.
	This yields
	\begin{align*}
		\dual{J'_y(\bar z)}{y-\bar y}_{\YY}
		&=
		\dual{S^*J'_y(\bar z)}{u-\bar u}_{\UU}
		\\
		&=
		\dual{-S^*A^*\bar\lambda^p}{u-\bar u}_{\UU}
		=
		-\dual{B^*\bar\lambda^p}{u-\bar u}_{\UU}, 
	\end{align*}
	and
	\begin{align*}
		\dual{J'_y(\bar z)}{y-\bar y}_{\YY} + \dual{J'_u(\bar z)}{u-\bar u}_{\UU}
		=
		\dual{J'_u(\bar z)-B^*\bar\lambda^p}{u-\bar u}_{\UU}
	\end{align*}
	follows.
	We note that $J'_u$ is Lipschitz continuous as $F'_u$ is assumed to be Lipschitz continuous while $f'_u$ is linear.
	Thus, taking the above and \cref{prop_parametric_penalized_1} into account, we find a constant $C>0$ such that
	\begin{align*}
		0
		&\leq
		\dual{J'_u(\bar z) - B^*\bar\lambda^p}{u-\bar u}_{\UU}
		\\
		&=
		\dual{J'_y(\bar z)}{y-\bar y}_{\YY} + \dual{J'_u(\bar z)}{u-\bar u}_{\UU}
		\\
		&=
		\dual{J'_u(\bar z) -J'_u(z) + J'_u(z) - B^*\lambda^p + B^*(\lambda^p-\bar\lambda^p)}{u-\bar u}_{\UU}
		\\
		&\leq
		\dual{J'_u(\bar z)-J'_u(z)}{u-\bar u}_{\UU} + \dual{B^*(\lambda^p-\bar\lambda^p)}{u-\bar u}_{\UU}
		\\
		&\leq
		C\norm{x-\bar x}^2.
	\end{align*}
	Thus, by \eqref{prob_cont_diff_1} and \eqref{eq:modified_oo}, we obtain
	\begin{align*}
		&\frac{\phi_\alpha(x)-\phi_\alpha(\bar x) - \dual{J'_x(\bar z)}{x-\bar x}}{\norm{x-\bar x}}
		\\
		&\qquad
		=
		\frac{\dual{J'_y(\bar z)}{y-\bar y}_{\YY} + \dual{J'_u(\bar z)}{u-\bar u}_{\UU}}{\norm{x-\bar x}}
		+
		\frac{\oo(\norm{z-\bar z}_{\ZZ})}{\norm{x-\bar x}}
		\to 0
		\quad\text{as}\quad x\to\bar x,
	\end{align*}
	showing that $\phi_\alpha$ is Fr\'{e}chet differentiable at $\bar x$ with Fr\'{e}chet derivative
	\begin{align*}
		\phi'_\alpha(\bar x) 
		&=
		J'_x(\bar z)
		\\
		&=
		F'_x(\bar x,\bar y,\bar u) + \alpha(E^*(E\bar x - \bar y) - \varphi'(\bar x))
		\\
		&=
		F'_x(\bar x,\bar y,\bar u) + \alpha\,E^*(y(\bar x) - \bar y)
		\\
		&=
		F'_x(\bar x,y_\alpha(\bar x),u_\alpha(\bar x)) +\alpha\,E^*(y(\bar x) - y_\alpha(\bar x)),
	\end{align*}
	see \cref{prop:smooth_value_function}.
	Noting that $F'_x$ and all appearing solution operators are continuous,
	see \cref{Prop_llp_1,prop_parametric_penalized_1},
	we conclude that $\phi_\alpha$ is even continuously Fr\'{e}chet differentiable.
	Finally, using the above formula for $\phi'_\alpha$
	and Lipschitz continuity of all appearing solution operators,
	see \cref{Prop_llp_1,prop_parametric_penalized_1} again,
	it is clear that $\phi'_\alpha$ is Lipschitz continuous
	whenever $F'_x$ possesses this property. 
\end{proof}

To close the subsection, let us mention that, based on the formula for $\phi'_\alpha$
derived in \cref{prop_parametric_penalized_2}
as well as our comments following \cref{Prop_llp_1,prop_parametric_penalized_1},
one could state a Lipschitz constant of $\phi'_\alpha$ in terms of initial problem data
if only Lipschitz constants of $F'_x$, $F'_y$, and $F'_u$ are accessible.
However, as this information cannot be exploited later on in a beneficial way,
we omit a detailed presentation which would make the paper much more technical.

\subsection{Construction of the algorithm}\label{sec:algo}

We now concentrate on the identification of stationary points of \eqref{penalized_problem}
with a focus on feasibility for \eqref{main_problem}.
Therefore, we solve a system
comprising the feasibility conditions of \eqref{penalized_problem}, the conditions from \eqref{eq:optimality-system_1},
as well as \eqref{eq:optimality_system_Xad} and \eqref{eq:optimality_system_Uad}.
Noting that \eqref{eq:optimality_system_Xad} and \eqref{eq:optimality-system_1x}
can be combined to eliminate the dual variable $\lambda^x$,
this leads to the system
\begin{align*}
	0
	&\in 
	\left\{
	F'_x(\bar{x}, \bar{y},  \bar u)
	+ 
	\alpha t\left[ f_x'(\bar{x},\bar{y},\bar{u}) - E^{*}\left( E\bar x - y(\bar x)\right) \right]
	\right\}
	+
	N_{X_\ad}(\bar x),
	\\
	0
	&=
	F'_y(\bar{x}, \bar{y},  \bar u) + A^* \lambda^p+\alpha t f'_y(\bar x,\bar y,\bar u)
	,
	\\
	0
	&=
	F'_u(\bar{x}, \bar{y},  \bar u) -B^*\lambda^p+ \lambda^u+\alpha tf'_u(\bar x,\bar y,\bar u)
	,
	\\
	0
	&=
	A\bar y-B\bar u, 
	\\
	\lambda^u&\in N_{U_\ad}(\bar u),
	\\
	t
	&\in 
	\partial\max\{\cdot,0\}(f(\bar x,\bar y,\bar u)-\varphi(\bar x)-\varepsilon).
\end{align*}
Note that the first and fifth condition implicitly require 
$\bar x\in X_\ad$ and $\bar u\in U_\ad$ by definition of the normal cone.
The unknowns in this system are
\[
	(\bar x,\bar y,\bar u,\lambda^p,\lambda^u,t)
	\in 
	\R^n\times\YY\times\UU\times\YY\times\UU\times\R.
\]

In \cref{alg:stat_pen_alg}, we suggest an iterative process which follows the ideas
drafted at the beginning of \cref{sec:penalized_subproblem}.
We automatically abort if a feasible point of \eqref{main_problem} is identified
as the approach, at its core, is a penalty method.
Hence, we keep on iterating if the current iterate is
infeasible to \eqref{main_problem},
so that $t=1$ can be used in the stationarity conditions of \eqref{penalized_problem} 
as mentioned in \cref{rem:relation_stationarity}.

\begin{algorithm}[Penalty-type solution method for \eqref{main_problem}]\leavevmode
	\label{alg:stat_pen_alg}
	\begin{algorithmic}[1]
	\Require $x_0\in X_\ad$, 
	 $\alpha_0>0$, 
	 $\gamma>1$, $\eta>0$, $\kappa\in (0,1)$
	
	\State Set $V_{-1}:=\infty$.
	
	\For{$k=0,1,\ldots$}
	
	\parState{Given $x_k$ and $\alpha_k$, find
	$(y_k,u_k,\lambda^p_k,\lambda^u_k)$ such that
	\begin{equation}\label{item:simple_subsystem}
	\begin{aligned}
	0
	&=
	F'_y(x_k,y_k,u_k)
	+ A^*\lambda^p_k
	+ \alpha_k f'_y(x_k,y_k,u_k)
	,
	\\
	0
	&=
	F'_u(x_k,y_k,u_k)
	- B^*\lambda^p_k
	+ \lambda^u_k
	+ \alpha_kf'_u(x_k,y_k,u_k)
	,
	\\
    0
    &=
    Ay_k-Bu_k
    ,
    \\
    \lambda^u_k &\in N_{U_\ad}(u_k).
	\end{aligned}
	\end{equation}
	}
	\parState{Compute $y(x_k)$, $\varphi(x_k)$, and set $V_k:=f(x_k,y_k,u_k)-\varphi(x_k)-\varepsilon$.}

	\If  {$V_k\le0$} \label{item:termination}
                       \State \Return \((x_k,y_k,u_k)\)
        \Else \label{item:update_alphak} 
               \If  {$V_k>\kappa V_{k-1}$}
         	\parState{Set $\alpha_{k+1}:=\gamma\alpha_k$.}       
           \Else  
         	\parState{Set $\alpha_{k+1}:=\alpha_k$.}
	\EndIf
          \EndIf
	\parState{Set
	\begin{align*}
	w_k
	&:=
	F'_x(x_k,y_k,u_k)
	+
	\alpha_{k+1}
	\left(
	f'_x(x_k,y_k,u_k)
	-
	E^*(Ex_k-y(x_k))
	\right),
	\\
	x_{k+1}
	&:=
	\proj_{X_\ad}
	\left(
	x_k
	-
	\eta w_k
	\right).
	\end{align*}
	}\label{item:update_xk}
	
	\EndFor
\end{algorithmic}
\end{algorithm}

In order to address the convergence analysis of \cref{alg:stat_pen_alg} properly,
additional assumptions are needed which are stated subsequently.

\begin{assumption}\label{Ass_cvgce}
\hfill
\begin{enumerate}
\item\label{item:Ass_F_lower_bounded} The function $F$ is lower bounded on $X_\ad\times S(U_\ad)\times U_{\ad}$.
\item\label{item:Ass_Lipschitzness} The partial Fr\'{e}chet derivatives $F'_x$, $F'_y$, and $F'_u$ are Lipschitz continuous.
\end{enumerate}
\end{assumption}

Before concentrating on the convergence analysis of \cref{alg:stat_pen_alg},
let us state some remarks about it.

\begin{remark}\label{rem:alg}	
	\hfill
	\begin{enumerate}
		\item In each iteration of \cref{alg:stat_pen_alg},
			the lower-level problem \eqref{Lower_level_problem} has to be solved just once
			as we do not exploit any linesearch strategies involving the optimal value function $\varphi$.
		\item\label{item:alg_well_defined}
			We note that system \eqref{item:simple_subsystem} corresponds to the (necessary and sufficient) optimality conditions of 
			the parametric optimization problem \hyperref[parametric_penalized_problem]{\textup{(PP$_{\alpha_k}(x_k)$)}},
			and can be interpreted as a square mixed-nonlinear complementarity system.
			According to \cref{prop_parametric_penalized_1}, this system possesses a uniquely determined primal-dual solution
			which depends in a Lipschitzian way on the parameter $x_k$.
			Hence, \cref{alg:stat_pen_alg} is well-defined.
		\item\label{item:alg_finite_termination} 
			In the case where \cref{alg:stat_pen_alg} terminates finitely 
			due to \hyperref[item:termination]{Line 5} in iteration $k\in\N$,
			it returns a feasible point of \eqref{main_problem} which satisfies, in a certain sense,
			approximate stationarity conditions associated with \eqref{main_problem}, see \cref{prop:NOC_IOCeps}. 
			The conditions from \eqref{item:simple_subsystem} correspond to 
			the conditions \eqref{eq:optimality_system_y} and \eqref{eq:optimality_system_u},
			with $\alpha_k$ playing the role of $\lambda^\textup{vf}$, \eqref{eq:optimality_system_Uad},
			and feasibility for \hyperref[Lower_level_problem]{\textup{(LL$(x_k)$)}}.
			Whenever $f(x_k,y_k,u_k)-\varphi(x_k)-\varepsilon<0$, 
			the associated complementarity slackness condition \eqref{eq:optimality_system_vf} is, however, violated. 
			By definition of $w_{k-1}$ and $x_{k}$ 
			in \hyperref[item:update_xk]{Line 14} in the previous iteration $k-1$
			and the fact that $N_{X_\ad}(x_k)$ is a cone, 
			we have
			\begin{align*}
				- \frac{x_k-x_{k-1}}{\eta}
				\in
				\{w_{k-1}\}
				+ 
				N_{X_\ad}(x_k).
			\end{align*}
			If $x_k$ and $x_{k-1}$ are sufficiently close together, 
			then the same holds true for $(y_k,u_k)$ and $(y_{k-1},u_{k-1})$ according to \cref{prop_parametric_penalized_1}
			whenever \cref{Ass_cvgce}\,\ref{item:Ass_Lipschitzness} is valid,
			and this then would also ensure that $w_k$ and $w_{k-1}$ are close together,
			i.e.,
			\eqref{eq:optimality_system_x} and \eqref{eq:optimality_system_Xad} would hold approximately
			in this situation,
			see \cref{Prop_llp_1} as well. 
			Let us emphasize that $x_k-x_{k-1}\to 0$ will be shown later on in \cref{lem:behavior_x}.
	\end{enumerate}
\end{remark}

According to \cref{rem:alg}\,\ref{item:alg_finite_termination},
it remains to analyze the situation when \cref{alg:stat_pen_alg}
produces an infinite sequence.
This will be done in \cref{sec:ana}.

\subsection{Convergence analysis}\label{sec:ana}

Now, we are in position to analyze the behavior of \cref{alg:stat_pen_alg}
in more detail. To start, we show that penalty parameters in \cref{alg:stat_pen_alg} remain bounded.

\begin{lemma}\label{lem:penalty_parameters_bounded}
	Let $\{(x_k,y_k,u_k)\}_{k\in\mathbb{N}}$ be the sequence generated by \cref{alg:stat_pen_alg},
 	and let $\{\alpha_k\}_{k\in\N}$ be the associated sequence of penalty parameters.
 	Then $\{\alpha_k\}_{k\in\N}$ is bounded and, thus, converges.
\end{lemma}
\begin{proof}
We argue by contradiction. 
Suppose that the sequence $\{\alpha_k\}_{k\in\N}$ is unbounded. 
As \cref{alg:stat_pen_alg} does not terminate finitely,
we know that $f(x_k,y_k,u_k)-\varphi(x_k)-\varepsilon>0$ holds for each $k\in\N$.
Fix $k\in \N$.
By construction, there are multipliers $(\lambda^p_k,\lambda^u_k)\in\YY\times\UU$
such that $(y_k,u_k,\lambda^p_k,\lambda^u_k)$ solves the system \eqref{item:simple_subsystem}. 
Hence, for every pair $(y,u)\in\YY\times\UU$ such that $Ay=Bu$ and $u\in U_{\ad}$ hold,
we obtain
\begin{equation}\label{Eq_11}
	\begin{aligned}
	&\dual{F_y'(x_k,y_k,u_k) + \alpha_k f_y'(x_k,y_k,u_k)}{y-y_k}_{\YY} \\
	&\qquad
	+\dual{F_u'(x_k,y_k,u_k) + \alpha_k f_u'(x_k,y_k,u_k)}{u-u_k}_{\UU} \\
	&\qquad
	=
	\dual{-A^*\lambda^p_k}{y-y_k}_{\YY} + \dual{B^*\lambda^p_k- \lambda^u_k}{u-u_k}_{\UU} \\
	&\qquad
	=
	\dual{\lambda^p_k}{A(y_k-y)}_{\YY} +\dual{\lambda^p_k}{B(u-u_k)}_{\YY} - \dual{\lambda^u_k}{u-u_k}_{\UU}\\
	&\qquad
	=
 	\dual{\lambda^p_k}{(Ay_k-Bu_k) - (Ay-Bu)}_{\YY} - \dual{\lambda^u_k}{u-u_k}_{\UU} \\
 	&\qquad
 	=
 	- \dual{\lambda^u_k}{u-u_k}_{\UU} \geq 0,
\end{aligned}
\end{equation}
where the last inequality follows from $\lambda^u_k \in N_{U_\ad}(u_k)$.
Furthermore, there exist $(\bar y_k,\bar u_k)\in\YY\times\UU$ such that
$A\bar y_k-B\bar u_k=0$, $\bar u_k\in U_{\ad}$, and $\varphi(x_k) = f(x_k,\bar y_k,\bar u_k)$,
see \cref{Prop_llp_1}.
Hence, the above yields
\begin{equation}\label{eq:some_bound}
	\begin{aligned}
	&
	\dual{ F_y'(x_k,y_k,u_k) + \alpha_k f_y'(x_k,y_k,u_k)}{\bar y_k-y_k}_{\YY} \\
	&\qquad
	+
	\dual{F_u'(x_k,y_k,u_k) + \alpha_k f_u'(x_k,y_k,u_k)}{\bar u_k-u_k}_{\UU}
	\geq
	0.
	\end{aligned}
\end{equation}
Note that boundedness of $X_{\ad}$ and $U_{\ad}$ as well as linearity and continuity of
$S$ imply the boundedness of $\{(x_k,y_k,u_k)\}_{k\in \N}$ and $\{(\bar y_k,\bar u_k)\}_{k\in \N}$.
Hence, according to \cref{Ass_cvgce}\,\ref{item:Ass_Lipschitzness}, we find a constant $C>0$ such that
\[
	\dual{ F_y'(x_k,y_k,u_k)}{\bar y_k-y_k}_{\YY} 
	+
	\dual{F_u'(x_k,y_k,u_k)}{\bar u_k-u_k}_{\UU}
	\leq 
	C
\]
is valid for all $k\in \N$.
Exploiting, for each $k\in \N$, the convexity of $f(x_k,\cdot,\cdot)$ and \eqref{eq:some_bound}, we find
\begin{align*}
	&f(x_k,y_k,u_k) - f(x_k,\bar y_k,\bar u_k)\\
	&\qquad
	\leq
	-
	\dual{f_y'(x_k,y_k,u_k)}{\bar y_k-y_k}_{\YY} - \dual{f_u'(x_k,y_k,u_k)}{\bar u_k-u_k}_{\UU}
	\\
	&\qquad
	\leq
	\frac1{\alpha_k}\left(
		\dual{ F_y'(x_k,y_k,u_k)}{\bar y_k-y_k}_{\YY} 
		+
		\dual{F_u'(x_k,y_k,u_k)}{\bar u_k-u_k}_{\UU}
	\right)
	\leq
	\frac{C}{\alpha_k}.
\end{align*}
Hence, for each $k\in \N$, we have
\[
	0 
	<
	f(x_k,y_k,u_k) - \varphi(x_k)-\varepsilon
	=
	f(x_k,y_k,u_k) - f(x_k,\bar y_k,\bar u_k) - \varepsilon
	\leq
	\frac{C}{\alpha_k} - \varepsilon,
\]
and taking the limit $k\to\infty$ yields a contradiction as $\alpha_k\to\infty$ and $\varepsilon>0$.
Hence, the sequence $\{\alpha_k\}_{k\in\N}$ must be bounded.
As it is also monotonically nondecreasing, it must be convergent, 
verifying the claim.
\end{proof}

Our next result addresses the behavior of the upper-level decision variables in \cref{alg:stat_pen_alg}.

\begin{lemma}\label{lem:behavior_x}
	Let $\{(x_k,y_k,u_k)\}_{k\in\mathbb{N}}$ be the sequence generated by \cref{alg:stat_pen_alg}
 	such that \cref{Ass_cvgce} holds.
 	Then, whenever the parameter $\eta>0$ in \cref{alg:stat_pen_alg} is chosen sufficiently small, 
 	we have $x_{k+1}-x_k\to 0$.
\end{lemma}
\begin{proof}
	According to \cref{lem:penalty_parameters_bounded} and the updating rule for the penalty parameter, 
	there exists $\bar k\in\N$ such that  $\alpha_k = \alpha_{\bar k}=:\bar\alpha$ holds for all $k\in\N$ such that $ k\geq \bar k$. 
	Throughout the proof, we merely consider $k\in\N$ such that $k\geq\bar k$.
	
	Recall that $\phi_{\bar\alpha}\colon\R^n\to\R$, defined in \eqref{eq:yet_another_value_function}, denotes the optimal value function
	of the parametric optimization problem \hyperref[parametric_penalized_problem]{(PP$_{\bar \alpha}(x)$)}.
	In \cref{prop_parametric_penalized_2}, we verified that $\phi_{\bar\alpha}$ is Lipschitz continuously Fr\'{e}chet differentiable,
	due to \cref{Ass_cvgce}\,\ref{item:Ass_Lipschitzness},
	and that, for each $k$, we have
	\begin{align*}
		\phi'_{\bar\alpha}(x_k) 
		&=
		F'_x(x_k,y_k,u_k) + \bar\alpha E^*(y(x_k) - y_k)
		\\
		&=
		F'_x(x_k,y_k,u_k) + \bar\alpha \bigl(f'_x(x_k,y_k,u_k) - E^*(Ex_k - y(x_k))\bigr)
		\\
		&= 
		w_k.
	\end{align*}
	Hence, \hyperref[item:update_xk]{Line 14} of \cref{alg:stat_pen_alg} yields
	\[
		x_{k+1} = \proj_{X_{\mathrm{ad}}}(x_k - \eta\, \phi'_{\bar\alpha}(x_k)).
	\]
	Using \eqref{eq:char_proj}, we particularly find
	\[
		\dual{x_{k+1} - x_k + \eta\,\phi'_{\bar\alpha}(x_k)}{x_k - x_{k+1}}\ge 0,
	\]
	and the latter yields
	\[
		\dual{ \phi'_{\bar\alpha}(x_k)}{ x_k - x_{k+1} }
		\ge
		\frac{1}{\eta}\norm{x_{k+1}-x_k}^2.
	\]
	The Lipschitzness of $\phi'_{\bar\alpha}$ guarantees the existence of a constant $L>0$ such that
	\[
		\phi_{\bar\alpha}(x_{k+1})
		\le
		\phi_{\bar\alpha}(x_k)
		+
		\dual{\phi'_{\bar\alpha}(x_k)}{x_{k+1}-x_k}
		+
		\frac{L}{2}\norm{x_{k+1}-x_k}^2.
	\]
	Combining both estimates yields
	\[
		\phi_{\bar\alpha}(x_{k+1}) 
		\leq
		\phi_{\bar\alpha}(x_k)
		-
		\left(\frac1\eta-\frac L2\right)\norm{x_{k+1}-x_k}^2.
	\]
	Hence, whenever $\eta<2/L$ is valid, $c:=1/\eta - L/2$ is a positive constant,
	and the above guarantees
	\[
		c\norm{x_{k+1}-x_k}^2 \leq \phi_{\bar\alpha}(x_k) - \phi_{\bar\alpha}(x_{k+1}).
	\]
	Summing up these inequalities for $k=\bar k,\ldots,N$, where $N\in\N$ is chosen such that $N\geq\bar k$ holds, yields
	\[
		c\sum_{k=\bar k}^N\norm{x_{k+1}-x_k}^2
		\leq
		\phi_{\bar\alpha}(x_{\bar k}) - \phi_{\bar\alpha}(x_{N+1}).
	\]
	Combining \cref{Ass_cvgce}\,\ref{item:Ass_F_lower_bounded} with the definition of $\phi_{\bar\alpha}$ in \eqref{eq:yet_another_value_function},
	we see that $\phi_{\bar\alpha}$ is lower bounded. 
	Hence, we find a constant $\beta\in\R$, not depending on $N$, such that
	\[
		\forall N\in\N\colon\quad 
		N\geq\bar k 
		\quad\Longrightarrow\quad
		c\sum_{k=\bar k}^N\norm{x_{k+1}-x_k}^2
		\leq
		\phi_{\bar\alpha}(x_{\bar k})-\beta. 
	\]
	Taking the limit $N\to\infty$ and observing that $\bar k$ is a finite integer, we end up with
	\[
		c\sum_{k=0}^\infty\norm{x_{k+1}-x_k}^2 < \infty.
	\]
	Due to $c>0$, $x_{k+1}-x_k\to 0$ follows.
\end{proof}

The proof of \cref{lem:behavior_x} indicates that $\eta<2/L$ has to hold in
\cref{alg:stat_pen_alg} in order to ensure the convergence $x_{k+1}-x_k\to 0$,
where $L>0$ is a Lipschitz constant of $\phi'_{\bar\alpha}$.
However, as $\bar\alpha>0$ is not known a priori,
this information is of limited practical use,
and it remains to identify suitable values of $\eta$, exemplary,
via a parameter study.

Now, we are in position to prove a subsequential convergence result regarding \cref{alg:stat_pen_alg}. 

\begin{theorem}\label{Theorem_convergence}
 	Let $\{(x_k,y_k,u_k)\}_{k\in\mathbb{N}}$ be the sequence generated by \cref{alg:stat_pen_alg}
 	such that \cref{Ass_cvgce} holds.
 	Then the sequence $\{(x_k, y_k, u_k)\}_{k\in\N}$  possesses a weakly convergent subsequence. 
	Moreover, whenever the parameter $\eta>0$ in \cref{alg:stat_pen_alg} is chosen sufficiently small,
	each weak accumulation point of $\{(x_k, y_k, u_k)\}_{k\in\N}$ is a stationary point of \eqref{main_problem}.
\end{theorem}
\begin{proof}
	To start, we note that the sequences $\{x_k\}_{k\in\N}\subset X_\ad$ and $\{u_k\}_{k\in\N}\subset U_\ad$
	remain bounded due to the boundedness of $X_\ad$ and $U_\ad$.
	Observing that $y_k=Su_k$ holds for each $k\in\N$,
	linearity and continuity of $S$ imply that $\{y_k\}_{k\in\N}\subset\YY$ remains bounded as well.
	As $\R^n\times\YY\times\UU$ is a Hilbert space,
	the bounded sequence $\{(x_k,y_k,u_k)\}_{k\in\N}$ possesses a weak accumulation point.
	Similar to the proof of \cref{lem:behavior_x},
	we fix $\bar k\in\N$ such that $\alpha_k=\alpha_{\bar k}=:\bar\alpha$ holds for all $k\in\N$
	such that $k\geq\bar k$.
	Note that, by construction of \cref{alg:stat_pen_alg}, we have $0\leq V_k\leq \kappa\,V_{k-1}$ for all $k\in\N$
	such that $k\geq\bar k$, and $V_k\to 0$ follows from $\kappa\in (0,1)$.
	
	Let $(\bar x,\bar y,\bar u)\in\R^n\times\YY\times\UU$ be a weak accumulation point of
	$\{(x_k,y_k,u_k)\}_{k\in\N}$, and let $K\subset\N$ be an infinite index set
	such that $(x_k,y_k,u_k)\rightharpoonup_K(\bar x,\bar y,\bar u)$.
	Weak sequential closedness of $X_\ad$ and $U_\ad$ guarantees $\bar x\in X_\ad$ and $\bar u\in U_\ad$,
	respectively.
	Noting that $\R^n$ is finite dimensional, $x_k\to_K\bar x$ follows.
	Furthermore, let $(\lambda^p_k,\lambda^u_k)\in\YY\times\UU$ be the multipliers which
	solve the system \eqref{item:simple_subsystem} along with $(y_k,u_k)$.
	For the purpose of simplicity, we assume that $k\geq\bar k$ holds for all $k\in K$.
	Then, for each $k\in K$, we note that
	\[
		y_k = y_{\bar\alpha}(x_k),
		\quad
		u_k = u_{\bar\alpha}(x_k),
		\quad
		\lambda^p_k = \lambda^p_{\bar\alpha}(x_k),
		\quad
		\lambda^u_k = \lambda^u_{\bar\alpha}(x_k)
	\]
	holds for the Lipschitz continuous functions $y_{\bar\alpha}$, $u_{\bar\alpha}$, $\lambda^p_{\bar\alpha}$, and $\lambda^u_{\bar\alpha}$
	which have been defined in \cref{prop_parametric_penalized_1}.
	Hence, we actually know the strong convergences $y_k\to_K\bar y$, $u_k\to_K\bar u$, $\lambda^p_k\to_K\lambda^p$, and $\lambda^u_k\to_K\lambda^u$
	for some pair $(\lambda^p,\lambda^u)\in\YY\times\UU$.
	For later use, set $\lambda^\textup{vf}:=\bar\alpha$.
	
	Taking the limit $k\to_K\infty$ in \eqref{item:simple_subsystem} while keeping closedness of the normal cone mapping $N_{U_\ad}(\cdot)$
	in mind, we find conditions \eqref{eq:optimality_system_y}, \eqref{eq:optimality_system_u}, and \eqref{eq:optimality_system_Uad}
	as well as $A\bar y = B\bar u$. 
	Furthermore, for the sequence $\{w_k\}_{k\in \N}$ defined 
	via \hyperref[item:update_xk]{Line 14} of \cref{alg:stat_pen_alg}, 
	we obtain the subsequential convergence
	\[
		w_k \to_K F'_x(\bar x,\bar y,\bar u) + \lambda^\textup{vf}\bigl(f'_x(\bar x,\bar y,\bar u) - E^*(E\bar x - y(\bar x))\bigr)=:-\lambda^x,
	\]
	see \cref{Prop_llp_1} as well,
	and this yields validity of \eqref{eq:optimality_system_x}.
	Due to $x_k\to_K\bar x$, \cref{lem:behavior_x} implies $x_{k+1}\to_K\bar x$.
	Hence, \hyperref[item:update_xk]{Line 14} and the continuity of $\proj_{X_\ad}$ imply  
	\[
		\bar x = \proj_{X_\ad}(\bar x + \eta\lambda^x).
	\] 
	This, due to \eqref{eq:char_proj}, is equivalent to $\eta\lambda^x\in N_{X_\ad}(\bar x)$.
	Due to $\eta>0$ and the fact that $N_{X_\ad}(\bar x)$ is a cone,
	\eqref{eq:optimality_system_Xad} has been verified.
	Finally, $V_k\to 0$ implies $f(\bar x,\bar y,\bar u)-\varphi(\bar x)-\varepsilon=0$
	which, on the one hand, shows that $(\bar x,\bar y,\bar u)$ is feasible to \eqref{main_problem}.
	On the other hand, this condition also yields validity of \eqref{eq:optimality_system_vf}.
	
	Wrapping things up, 
	we have shown that $(\bar x,\bar y,\bar u)$ is a stationary point of \eqref{main_problem}.
\end{proof}

Let us close this section with a final remark.

\begin{remark}
	\cref{Theorem_convergence} shows that whenever \cref{alg:stat_pen_alg} produces an infinite sequence,
	then its associated weak accumulation points are stationary points of \eqref{main_problem}.
	In fact, the proof of \cref{Theorem_convergence} reveals that each weak accumulation point 
	is already an accumulation point in terms of strong convergence.
	At its core, \cref{alg:stat_pen_alg} is a penalty method induced by the penalized problem \eqref{penalized_problem}
	that computes, roughly speaking, a sequence of stationary points associated with the penalized subproblem
	for potentially increasing values of the penalty parameter.
	\cref{lem:penalty_parameters_bounded} shows that finite values of the penalty parameter are already enough for this purpose.
	To some extent, this is not surprising as the penalty function used in \eqref{penalized_problem} is nonsmooth
	while Robinson's constraint qualification is valid at each feasible point of the original problem \eqref{main_problem},
	see \cref{prop:RCQ}.
	This is in line with classical observations regarding so-called exact penalization, 
	see \cite{di1988exactness} for a finite-dimensional perspective.
\end{remark}

\section{Numerical experiments and discussion}\label{sec:experiments}

We consider the inverse optimal control problem \eqref{eq:IOC_ref} 
which has been addressed in \cite[Section~6.2]{dempe2019solving} already.

Therein, $n:=2$ and, for the domain $\Omega:=(-1,1)\times(1,1)$, 
\[
	\YY := H^1_0(\Omega),
	\quad
	\UU := L^2(\Omega),
	\quad
	\VV := L^2(\Omega).
\]
We make use of the upper-level objective function $F$ given by
\[
	F(x,y,u)
	:=
	\frac12\norm{y-y_\textup{o}}_{L^2(\Omega)}^2 + \dual{(0.1,0.3)}{x}
\]
where $y_\textup{o} := 0.2\,e_1 + 0.3\,e_2$ holds for the functions $e_1,e_2\in L^2(\Omega)$ given by
\begin{align*}
	e_1(\omega)
	&:=
	10\exp\bigl(
	-5(\omega_1-0.7)^2
	-5(\omega_2-0.3)^2
	\bigr),
	\\
	e_2(\omega)
	&:=
	10\exp\bigl(
	-10(\omega_1+0.4)^2
	-10(\omega_2-0.5)^2
	\bigr) 
\end{align*}
for all $\omega=(\omega_1,\omega_2)\in\Omega$,
see \cref{fig:desired_state}.
Furthermore, we choose
\[
	X_{\ad}:=\operatorname{conv}\{(0,0),(1,0),(0,1)\}.
\]

\begin{figure}[ht]
\centering
\includegraphics[width=0.4\textwidth]{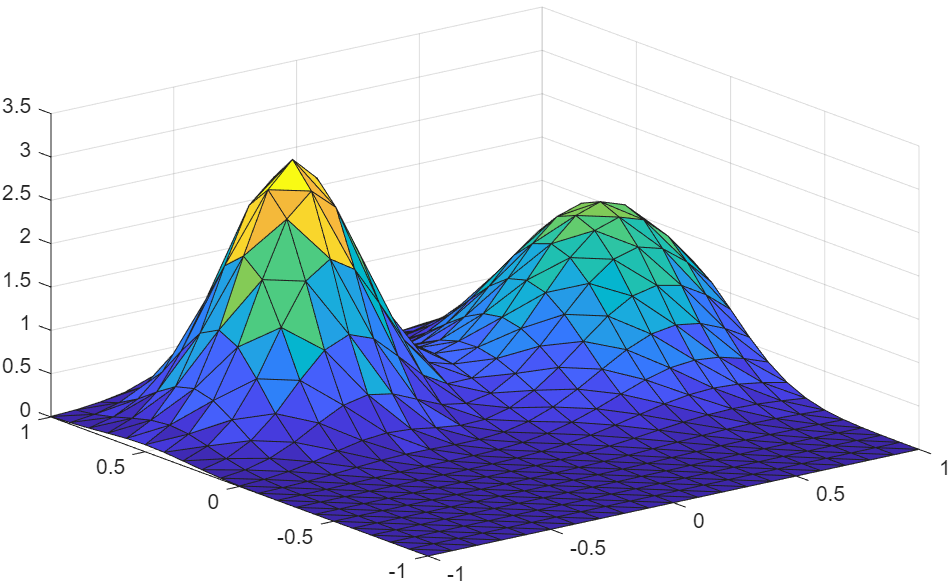}
\caption{Observed state $y_\textup{o}$.}
\label{fig:desired_state}
\end{figure}

To specify the subordinate lower-level problem \eqref{Lower_level_problem},
we set $Ex := x_1\,e_1 + x_2\,e_2$, $Ay:=-\Delta y$, and the role of $B$
is played by the natural embedding $L^2(\Omega)\hookrightarrow H^{-1}(\Omega)$.
Additionally, we exploit
\[
	U_\ad := \{u\in L^2(\Omega)\,|\,u_a \leq u \leq u_b\text{ a.e.\ on }\Omega\},
\]
where the control bounds $u_a,u_b\in L^2(\Omega)$ are constructed as described in \cite{dempe2019solving}.
The lower-level control regularization parameter $\sigma>0$ will be varied in our experiments,
where we will take on \eqref{main_problem} for varying values of the relaxation parameter $\varepsilon>0$.

Let us note that \cref{ass:setting} obviously holds for this particular instance of \eqref{eq:IOC_ref}.
\cref{Ass_cvgce}\,\ref{item:Ass_F_lower_bounded} is valid by definition of $F$ and the compactness of $X_\ad$,
and the quadratic structure of $F$ also implies that \cref{Ass_cvgce}\,\ref{item:Ass_Lipschitzness} is satisfied.

\subsection{On the discretized version of the algorithm}\label{sec:alg_discrete}

At the heart of \cref{alg:stat_pen_alg} is the numerical solution of the system \eqref{item:simple_subsystem}.
Suppressing the iteration index $k$, this system takes the particular form
\begin{equation}\label{eq:simple_system_ex}
	\begin{aligned}
	0
	&=
	y-y_\textup{o} + \alpha(y-Ex) - \Delta\lambda^p,
	\\
	0
	&=
	\alpha\sigma u - \lambda^p + \lambda^u,
	\\
	0
	&=
	-\Delta y - u,
	\\
	\lambda^u & \in N_{U_\ad}(u)
	\end{aligned}
\end{equation}
in the present situation,
and this is a square mixed-linear complementarity system.
The special structure of $U_\ad$ allows us to restate the last condition by means of $u\in U_\ad$ and
\[
	\forall\omega\in\Omega\colon\quad
	\lambda^u(\omega)
	\begin{cases}
		\leq 0 	&	u(\omega) = u_a(\omega),
		\\
		= 0		&	u_a(\omega) < u(\omega) < u_b(\omega),
		\\
		\geq 0	&	u(\omega) = u_b(\omega).
	\end{cases}
\]
Hence, introducing the measurable sets
\[
	\AA := \{\omega\in\Omega\,|\,u(\omega)\in\{u_a(\omega),u_b(\omega)\}\},
	\qquad
	\II := \Omega\setminus\AA,
\]
it essentially remains to determine $u$ on $\II$, the so-called inactive set, and $\lambda^u$ on $\AA$, the so-called active set.

We now discretize the above system \eqref{eq:simple_system_ex} with the aid of piecewise linear finite elements for $y$, $u$, $\lambda^p$, and $\lambda^u$
as suggested in \cite{dempe2019solving}. In order to ensure that the above characterization of the normal cone condition
transfers to the discretized setting, mass lumping for the control variable $u$ is exploited.

For simplicity of notation, the associated weight vectors are represented by
$y$, $u$, $\lambda^p$, and $\lambda^u$ again. 
Let $K$ and $M$ be the associated stiffness and (lumped) mass matrices, respectively,
and let $b(x,\alpha)$ be the load vector associated with $y_\textup{o} + \alpha Ex$.
Then the first three conditions in \eqref{eq:simple_system_ex} translate into
\begin{equation}\label{eq:discretized_system}
	\begin{aligned}
	b(x,\alpha)
	&=
	(1+\alpha)M y + K\lambda^p,
	\\
	0
	&=
	\alpha\sigma M u - M\lambda^p + M\lambda^u,
	\\
	0
	&=
	Ky - Mu.
	\end{aligned}
\end{equation}
Again, with a slight abuse of notation, we exploit $\mathcal A$ and $\mathcal I$ for the active and inactive index sets of the
discretized control as well as the active set decomposition
\[
	u = \begin{bmatrix} u_{\II} \\ u_{\AA} \end{bmatrix},
	\qquad
	\lambda^u = \begin{bmatrix} 0 \\ \lambda^u_{\AA} \end{bmatrix},
	\qquad
	M = \begin{bmatrix} M_{\II\II} & M_{\II\AA} \\ M_{\AA\II} & M_{\AA\AA} \end{bmatrix}.
\]
Plugging this into \eqref{eq:discretized_system}, the second equation decouples into
\begin{align*}
	0
	&=
	\alpha\sigma(M_{\II\II}u_\II + M_{\II\AA}u_\AA) - (M_{\II\II}\lambda^p_\II + M_{\II\AA}\lambda^p_\AA) + M_{\II\AA}\lambda^u_\AA,
	\\
	0
	&=
	\alpha\sigma(M_{\AA\II}u_\II + M_{\AA\AA}u_\AA) - (M_{\AA\II}\lambda^p_\II + M_{\AA\AA}\lambda^p_\AA) + M_{\AA\AA}\lambda^u_\AA,
\end{align*}
and the third equation therein is the same as
\[
	0 = Ky - M_{:\II}u_\II - M_{:\AA}u_\AA,
\]
where we used
\[
	M_{:\II} := \begin{bmatrix} M_{\II\II} \\ M_{\AA\II} \end{bmatrix},
	\qquad
	M_{:\AA} := \begin{bmatrix} M_{\II\AA} \\ M_{\AA\AA} \end{bmatrix}.
\]
Hence, we end up with the system
\begin{equation}\label{eq:discretized_system_reduced}
	\begin{bmatrix}
		b(x,\alpha)
		\\
		-\alpha\sigma M_{:\AA}u_\AA
		\\
		M_{:\AA}u_\AA
	\end{bmatrix}
	=
	\begin{bmatrix}
		(1+\alpha)M	&	\mathbb O	&	K	&	\mathbb O	
		\\
		\mathbb O	&	\alpha\sigma M_{:\II} & -M	&	M_{:\AA}
		\\
		K			& 	- M_{:\II}	&	\mathbb O	&	\mathbb O
	\end{bmatrix}
	\begin{bmatrix}
		y \\ u_\II \\ \lambda^p \\ \lambda^u_\AA
	\end{bmatrix}
	,
\end{equation}
where $\mathbb O$ is the all-zero matrix of appropriate dimensions.
Exploiting that $M$ is lumped while $\AA$ and $\II$ are disjoint,
one can use standard arguments in order to verify that the system matrix in \eqref{eq:discretized_system_reduced} is invertible.

Thus, solving \eqref{eq:simple_system_ex} reduces to the identification of $\AA$ and the numerical solution
of the linear system of equations \eqref{eq:discretized_system_reduced}.
According to \eqref{eq:char_proj}, we know that $\lambda^u\in N_{U_\ad}(u)$ is equivalent to
\[
	u = \proj_{U_\ad}(u + \vartheta\lambda^u)
\]
for each $\vartheta>0$.
Hence, to guess $\mathcal A$ in numerical practice in iteration $k\in\N_0$, we set $v_k := u_{k-1} + \vartheta \lambda^u_{k-1}$
and use
\[
	\mathcal A := \{\omega\in\Omega\,|\,v_k(\omega)\leq u_a(\omega) + \tau\} \cup \{\omega\in\Omega\,|\,v_k(\omega)\geq u_b(\omega) - \tau\}
\]
for yet another parameter $\tau>0$. 
In iteration $k=0$, we make use of $u_{-1}:=u_b$ and $\lambda^u_{-1}:=0$.

To close this subsection, let us note that one may also use a semismooth Newton method to solve \eqref{eq:simple_system_ex} in each iteration.
However, in our experiments, 
the elementary procedure drafted above already lead to convincing results 
which is why we relied on it.

\subsection{Results of experiments}
All numerical experiments were conducted on a Dell OptiPlex~7070 
equipped with an Intel\textsuperscript{\textregistered} Core\texttrademark~i5-9500T processor
(2.20~GHz), 16~GB RAM, running Microsoft Windows~11 Pro (64-bit).
We implemented \cref{alg:stat_pen_alg} in \texttt{MATLAB R2025b} using the parameter values
\[
	\alpha_0 := 0.005,  \quad \gamma := 1.3,\quad   \eta := 0.05,   \quad    \kappa := 0.9,
\]
and made use of the starting point $x_0:=(0,0)$.
In our experiments, we abort \cref{alg:stat_pen_alg} as soon as the termination condition 
from \hyperref[item:termination]{Line 5} is satisfied (up to machine precision).
The domain $\Omega$ is discretized by a Delaunay triangulation generated from a grid with mesh size $h=0.1$. 
The resulting mesh contains $441$ vertices and $800$ triangular elements,
and, hence, possesses the same specifications as the one exploited in \cite{dempe2019solving}
so that a fair comparison of the obtained results is possible.
Recall that piecewise linear finite elements are used to discretize all infinite-dimensional objects.
Linear systems are solved with the aid of \texttt{MATLAB}'s backslash operator for simplicity.
The parameters used for the approximation of the active set as described in \cref{sec:alg_discrete} are set to
\[
	\vartheta := 0.01,\quad \tau := 10^{-10}.
\]
Projections onto the polyhedron $X_\ad$ are computed using \texttt{MATLAB}'s \texttt{quadprog} routine.

\cref{fig:sigma_=_0.01_i,fig:sigma_=_0.01_ii} show the results produced by \cref{alg:stat_pen_alg} for $\sigma:=10^{-2}$ and  
different values of $\varepsilon$. 
In each row,
the first two plots present the state-control pair $(y_{\textup{alg}},u_{\textup{alg}})$ 
associated with \eqref{main_problem}, i.e., the final state-control pair returned by \cref{alg:stat_pen_alg}, 
while the last two plots show the state-control pair $(y_{\mathrm{ll}},u_{\mathrm{ll}})$ associated with 
\hyperref[Lower_level_problem]{(LL$(x_{\textup{alg}})$)}
for $x_{\textup{alg}}$ being the vector of the upper-level variables returned by \cref{alg:stat_pen_alg}.

\begin{figure}[ht]
\centering
\includegraphics[width=0.24\textwidth]{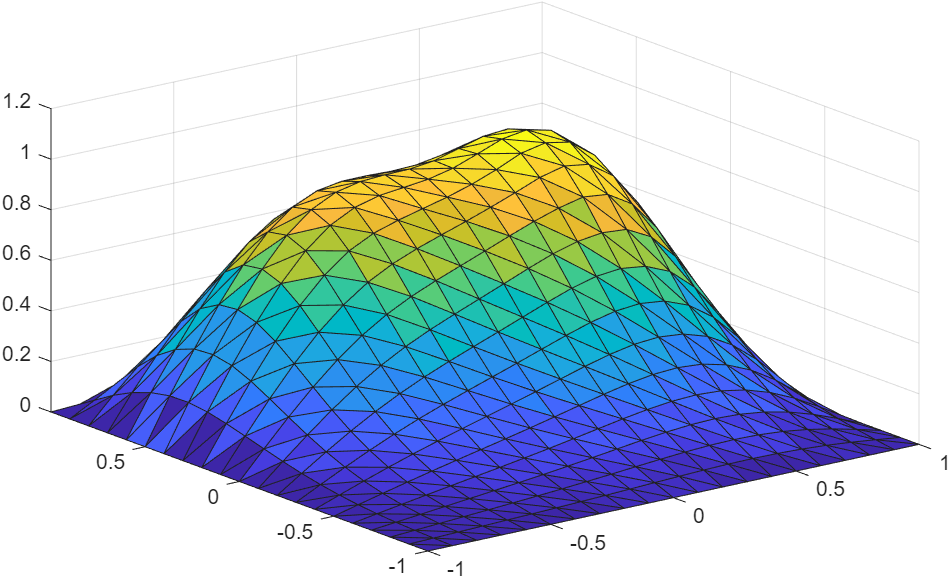}
\includegraphics[width=0.24\textwidth]{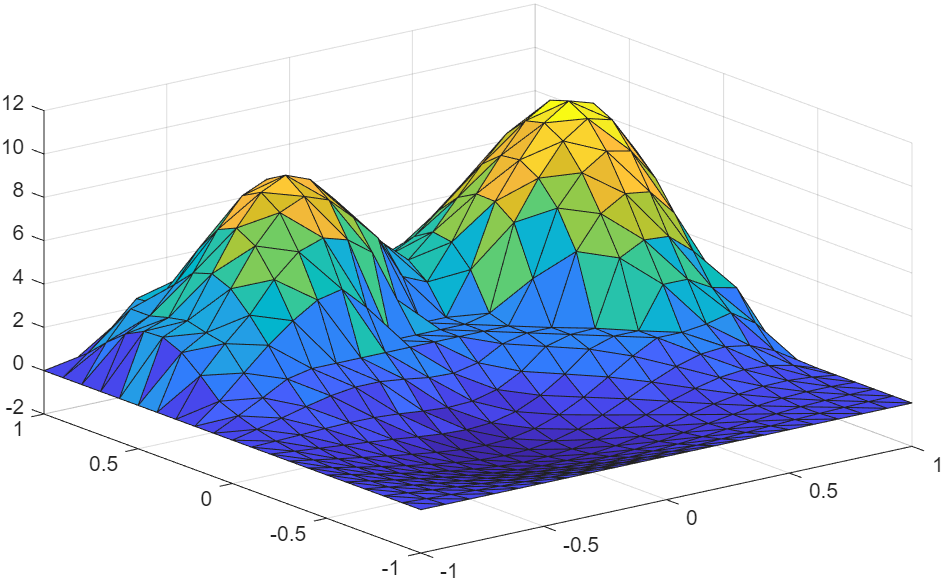}
\includegraphics[width=0.24\textwidth]{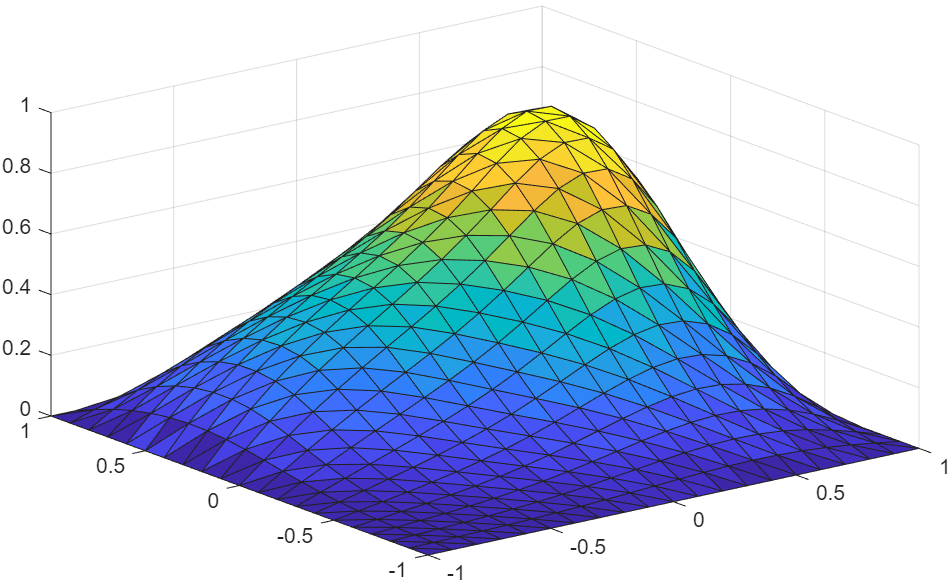}
\includegraphics[width=0.24\textwidth]{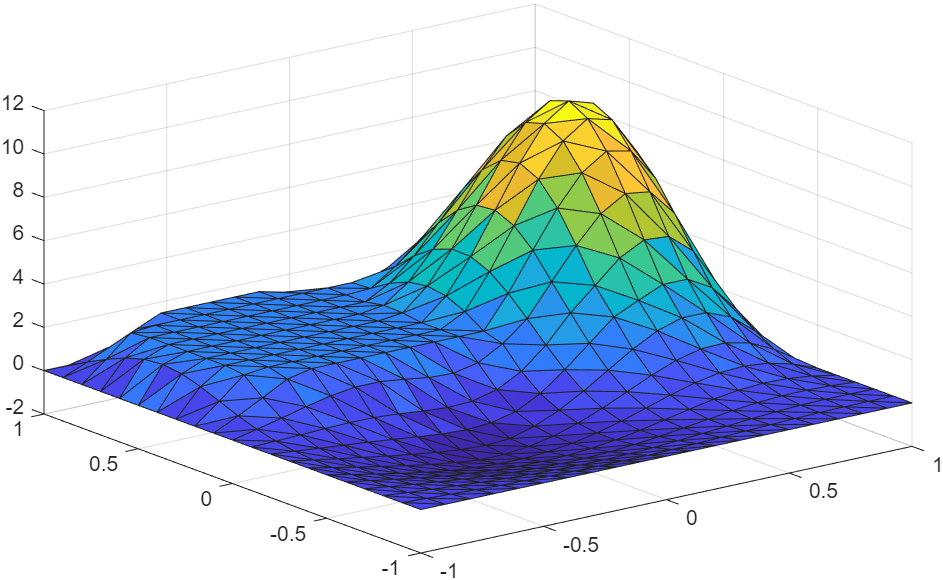}
\\
\includegraphics[width=0.24\textwidth]{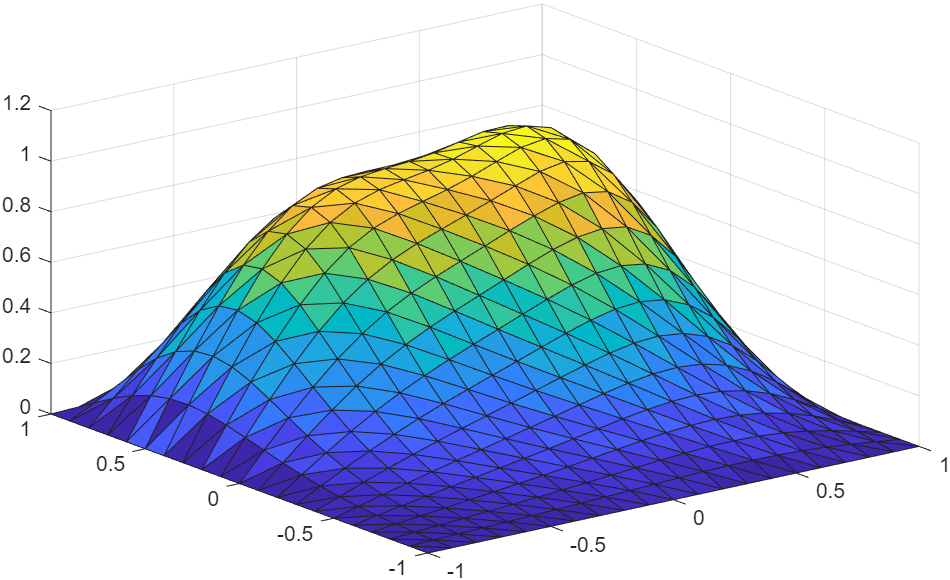}
\includegraphics[width=0.24\textwidth]{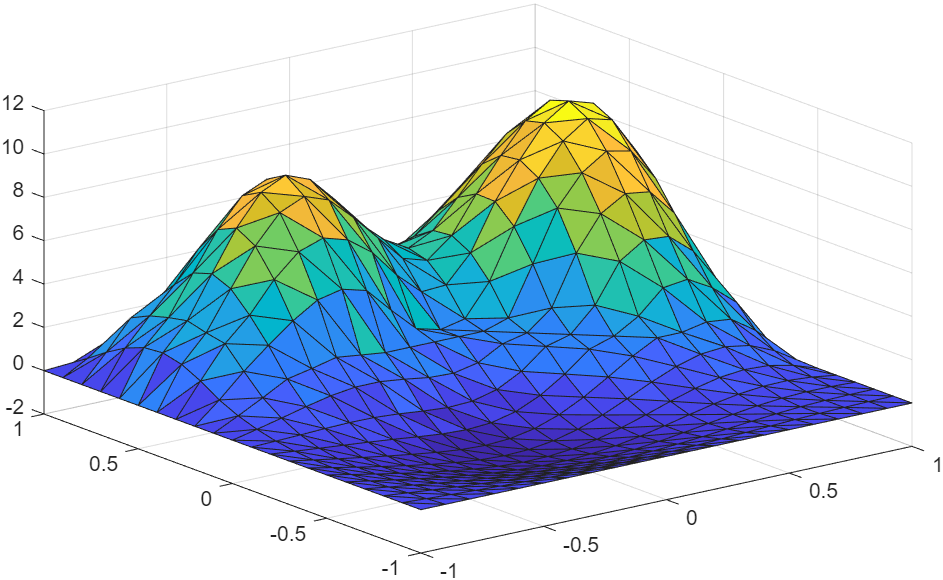}
\includegraphics[width=0.24\textwidth]{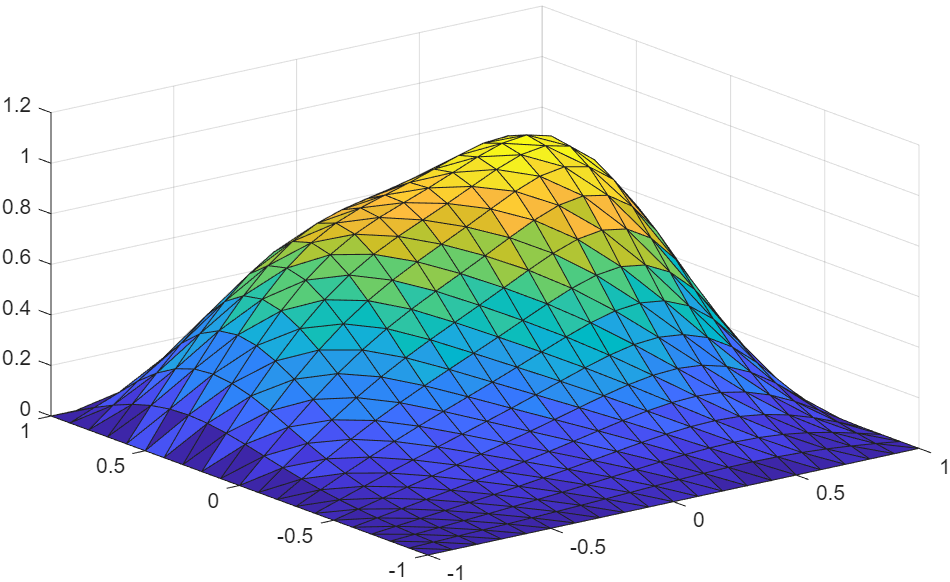}
\includegraphics[width=0.24\textwidth]{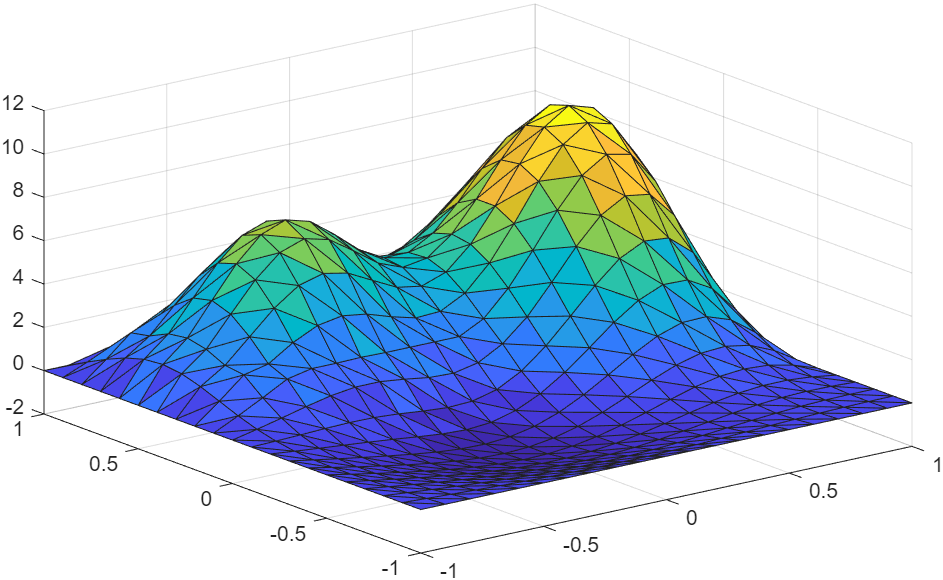}
\\
\includegraphics[width=0.24\textwidth]{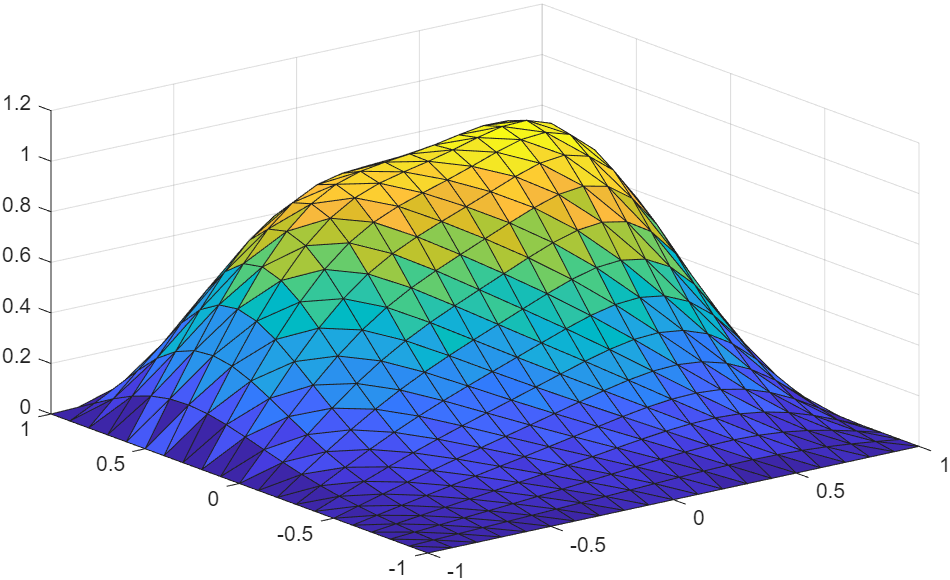}
\includegraphics[width=0.24\textwidth]{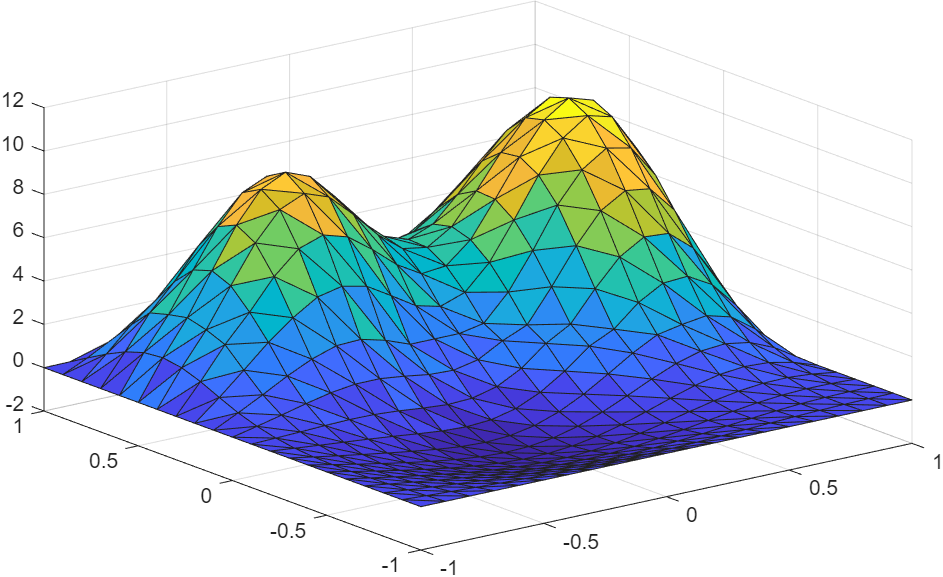}
\includegraphics[width=0.24\textwidth]{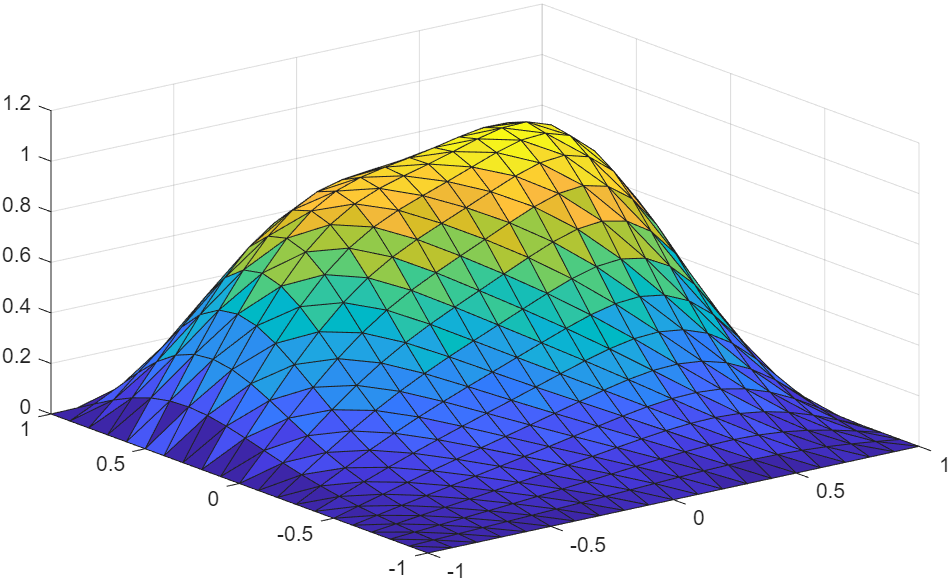}
\includegraphics[width=0.24\textwidth]{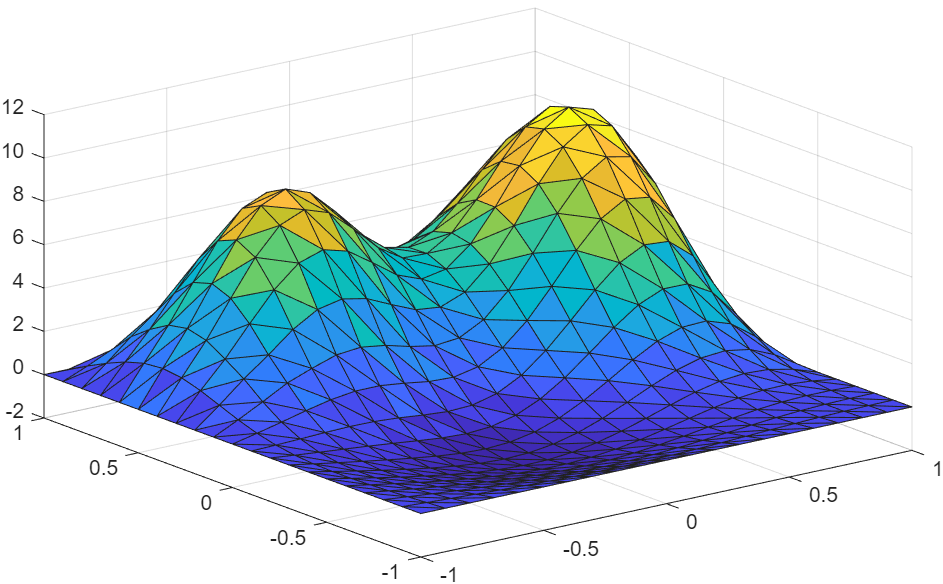}
\\
\includegraphics[width=0.24\textwidth]{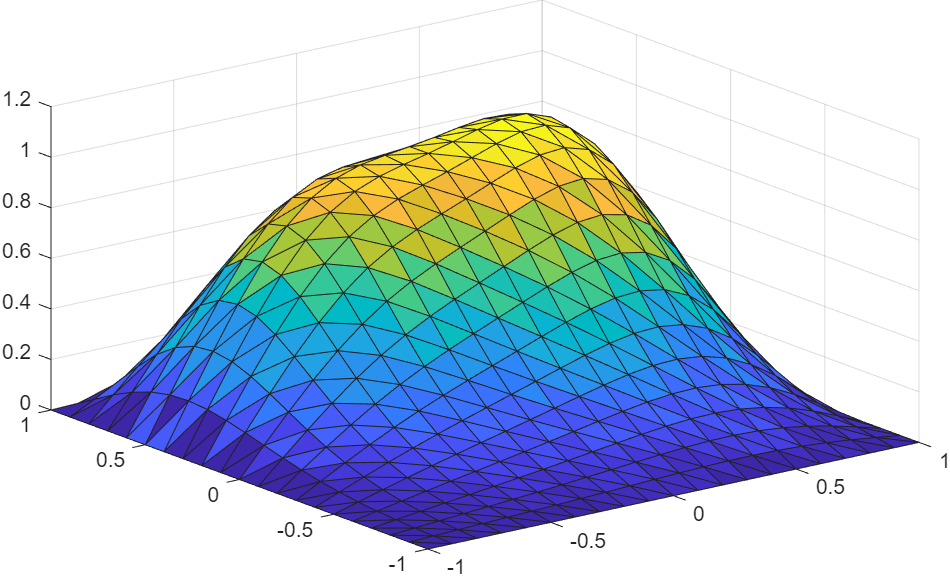}
\includegraphics[width=0.24\textwidth]{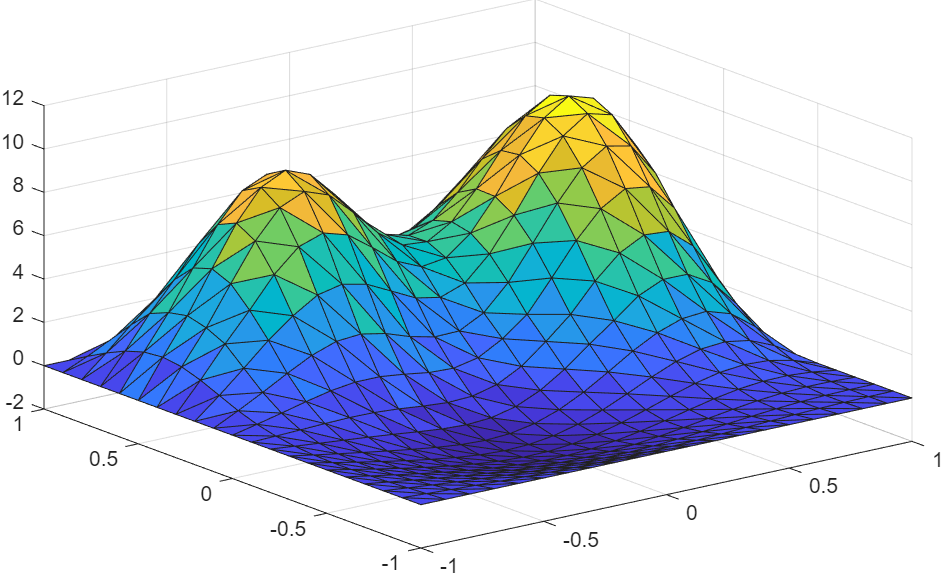}
\includegraphics[width=0.24\textwidth]{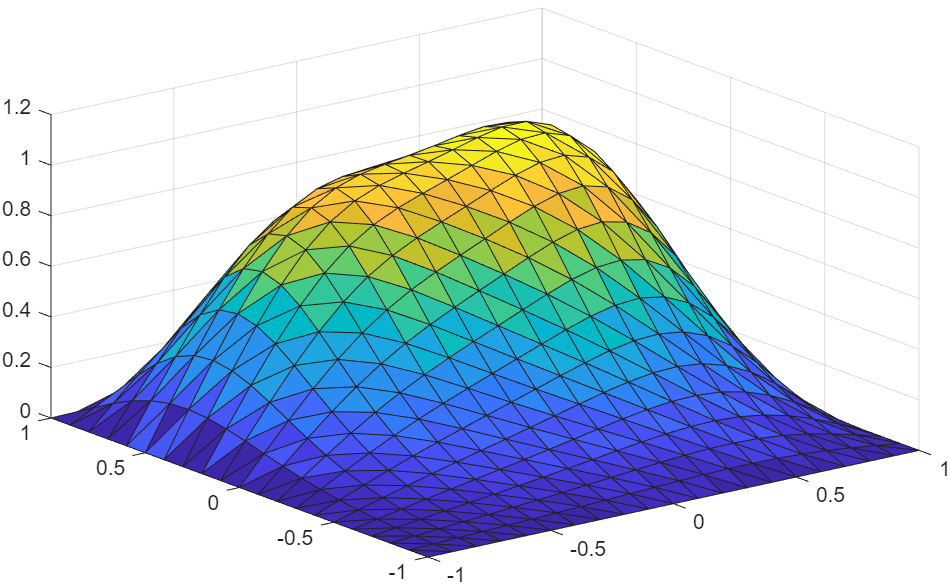}
\includegraphics[width=0.24\textwidth]{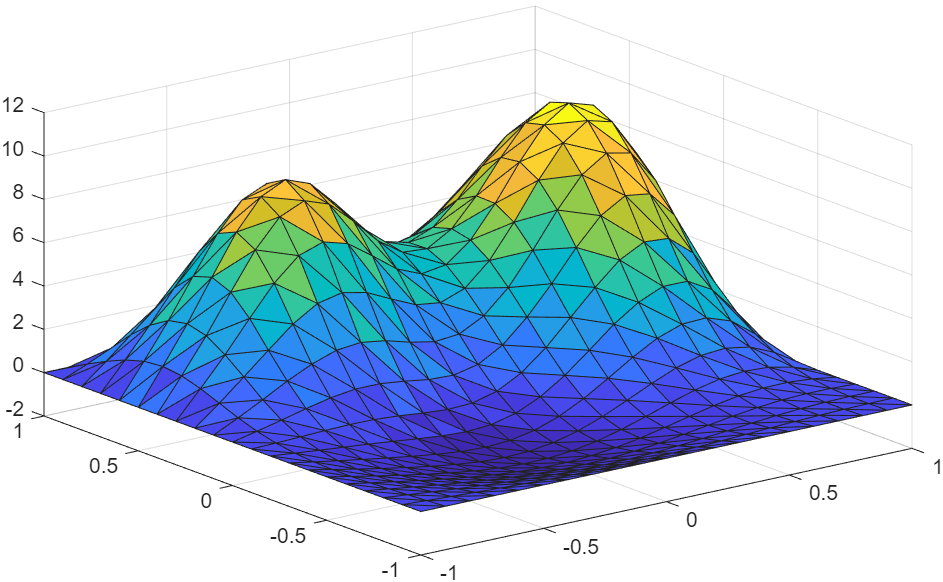}
\caption{State-control pairs for $\sigma := 10^{-2}$ and 
	$\varepsilon := 10^{-1},\ldots,10^{-4}$
	(from top to bottom);
	the first two columns show $(y_\textup{alg},u_\textup{alg})$,
	the last two columns show $(y_\textup{ll},u_\textup{ll})$.}
\label{fig:sigma_=_0.01_i}
\end{figure}

\begin{figure}[ht]
\centering
\includegraphics[width=0.24\textwidth]{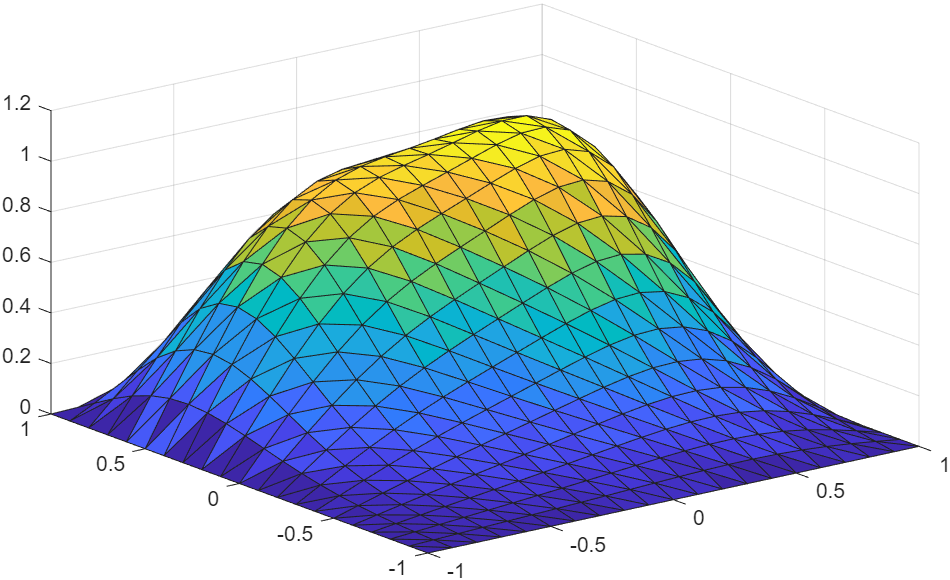}
\includegraphics[width=0.24\textwidth]{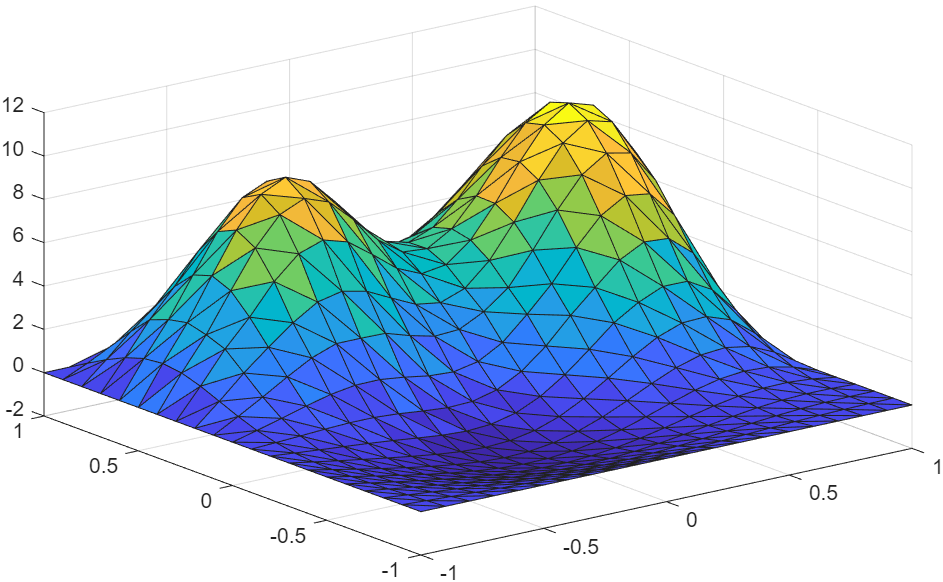}
\includegraphics[width=0.24\textwidth]{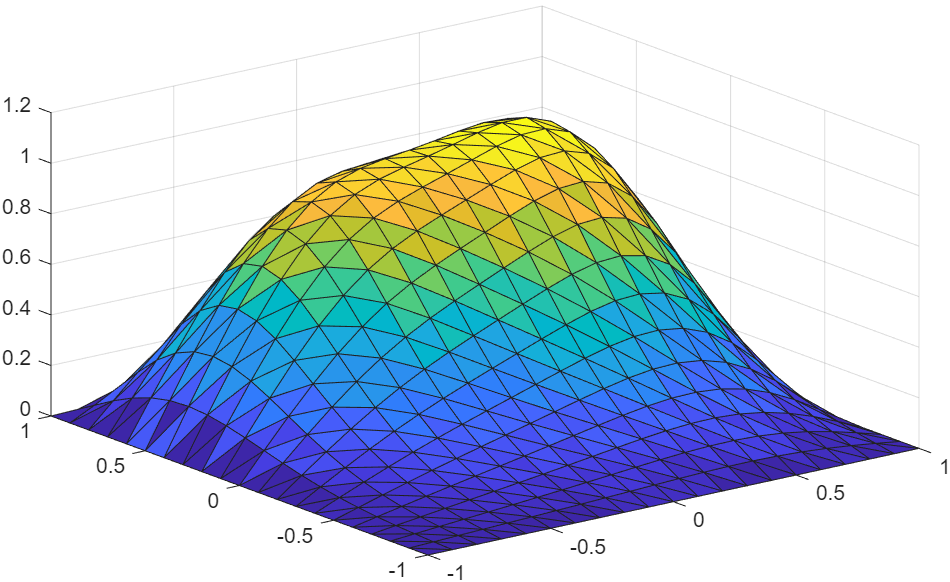}
\includegraphics[width=0.24\textwidth]{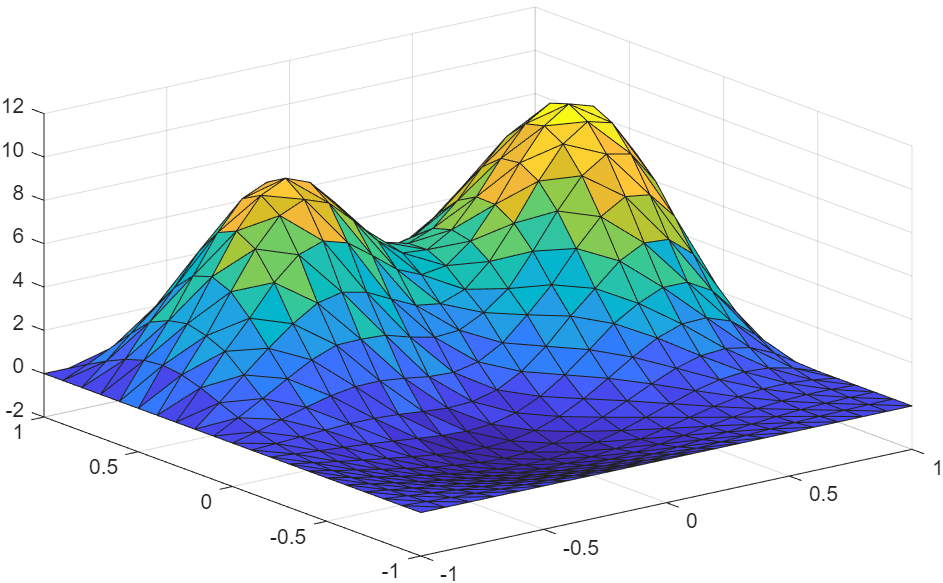}
\\
\includegraphics[width=0.24\textwidth]{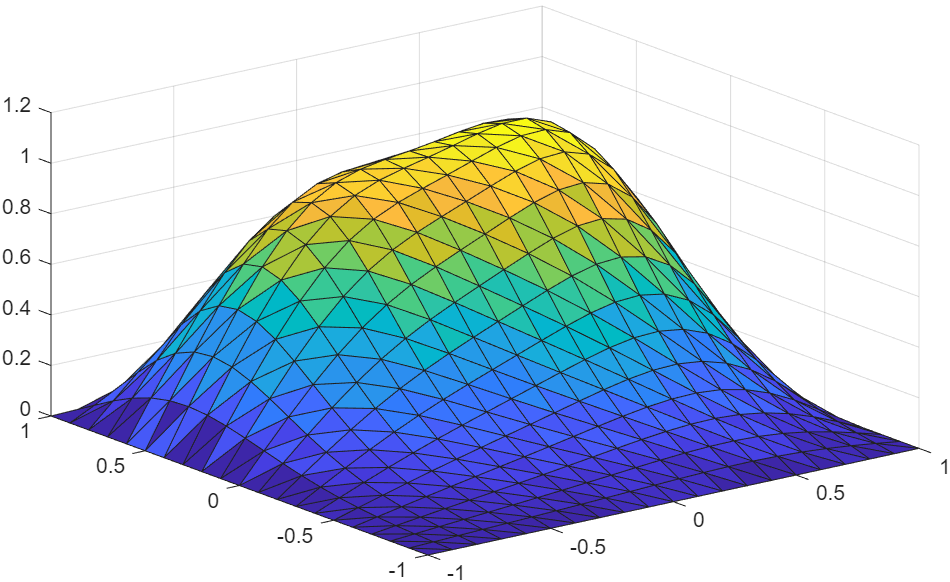}
\includegraphics[width=0.24\textwidth]{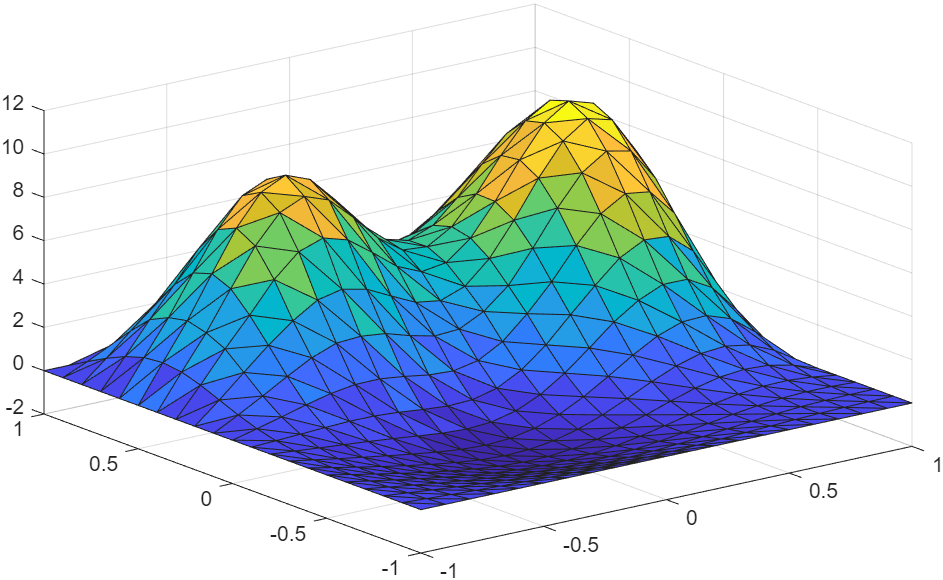}
\includegraphics[width=0.24\textwidth]{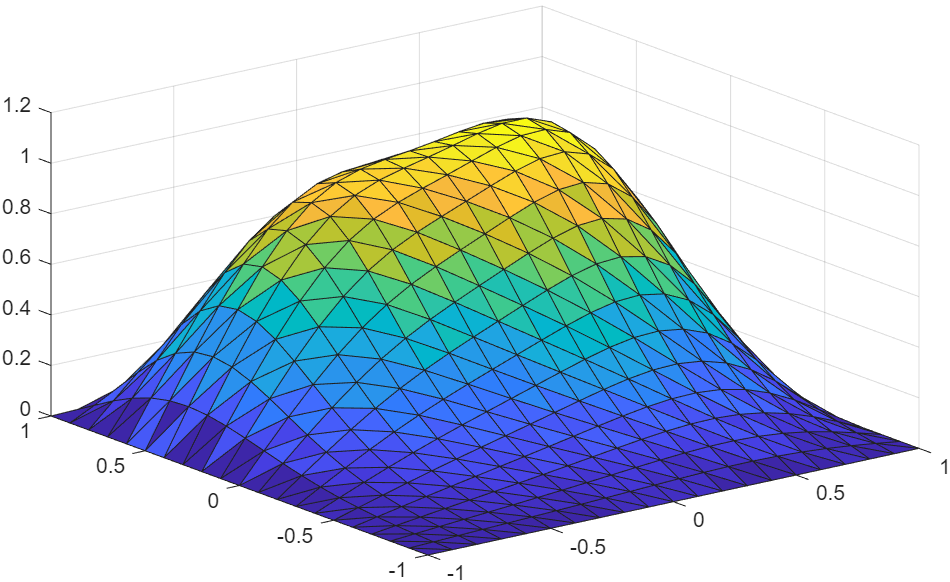}
\includegraphics[width=0.24\textwidth]{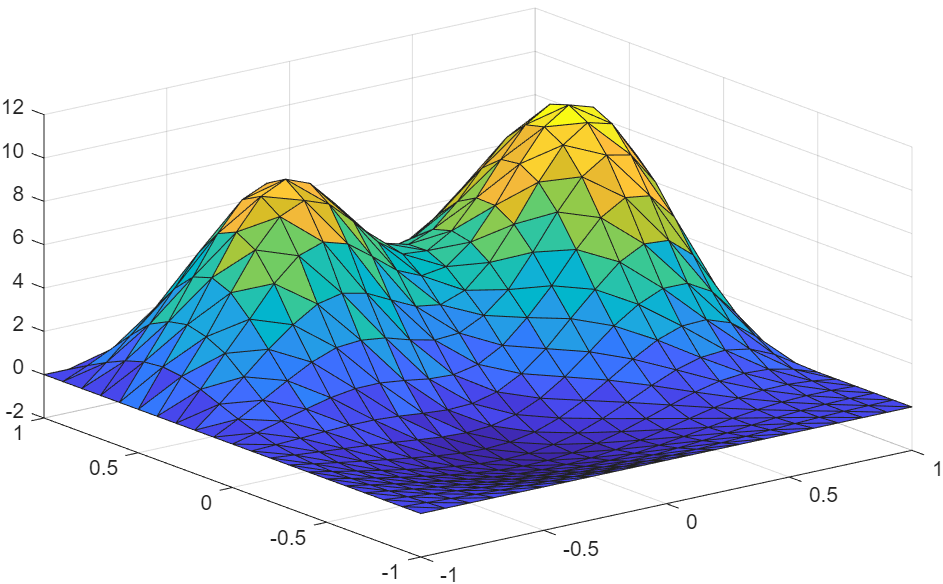}
\\
\includegraphics[width=0.24\textwidth]{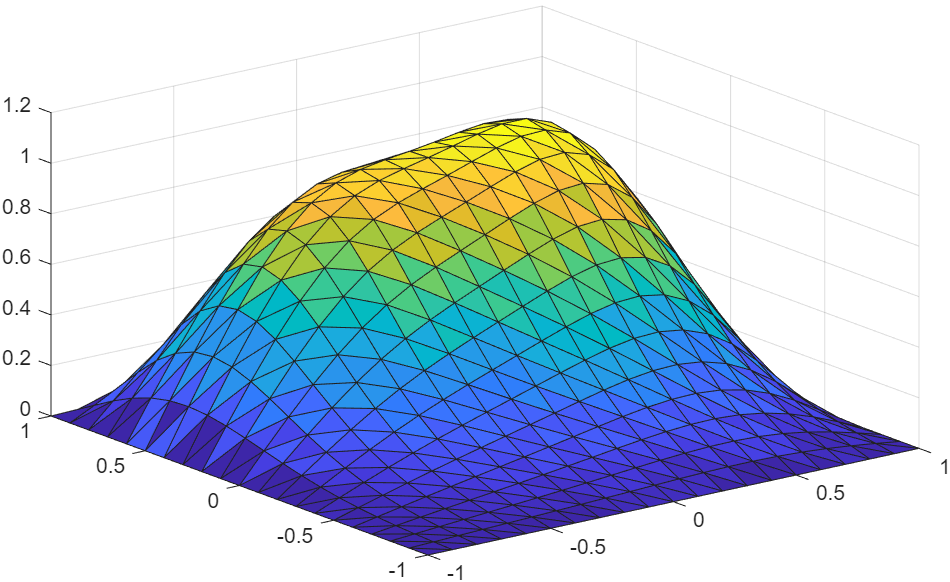}
\includegraphics[width=0.24\textwidth]{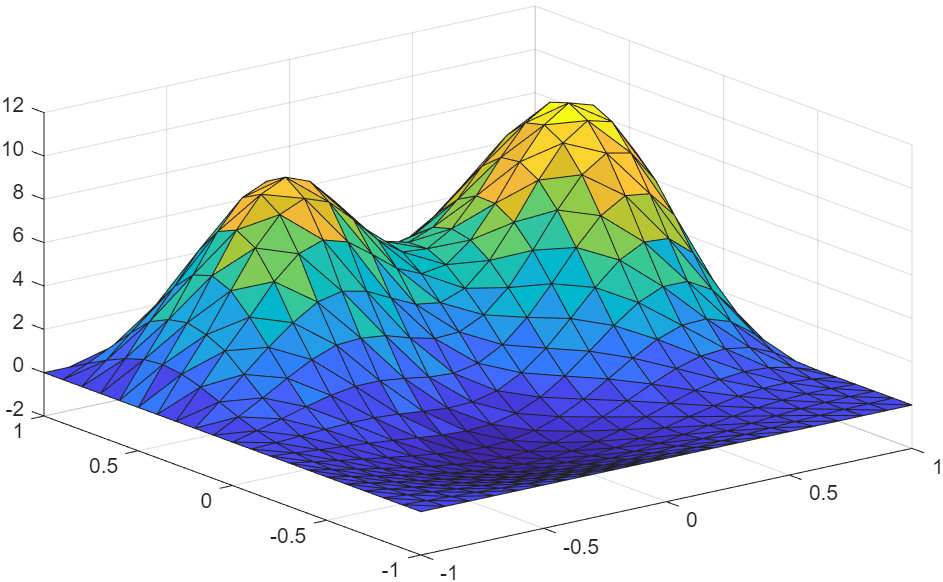}
\includegraphics[width=0.24\textwidth]{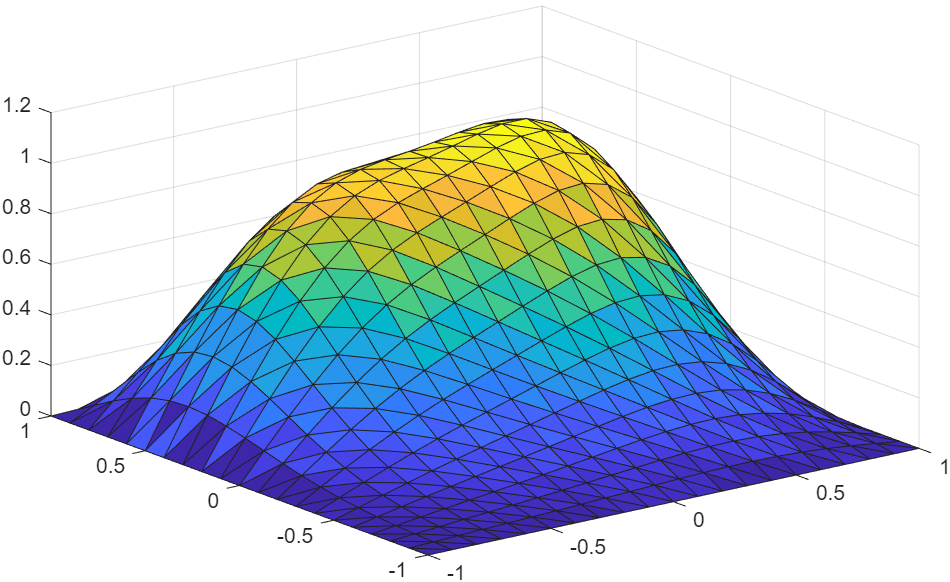}
\includegraphics[width=0.24\textwidth]{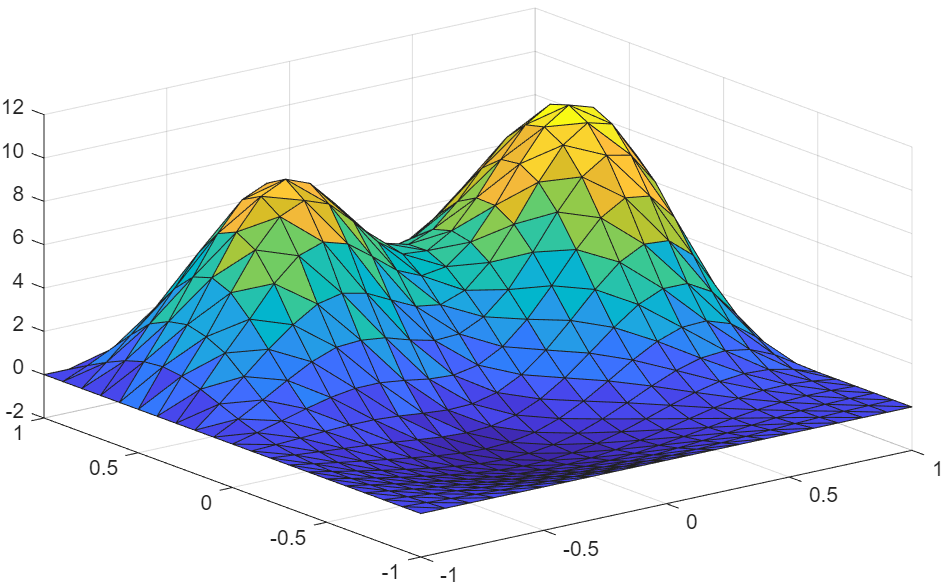}
\\
\includegraphics[width=0.24\textwidth]{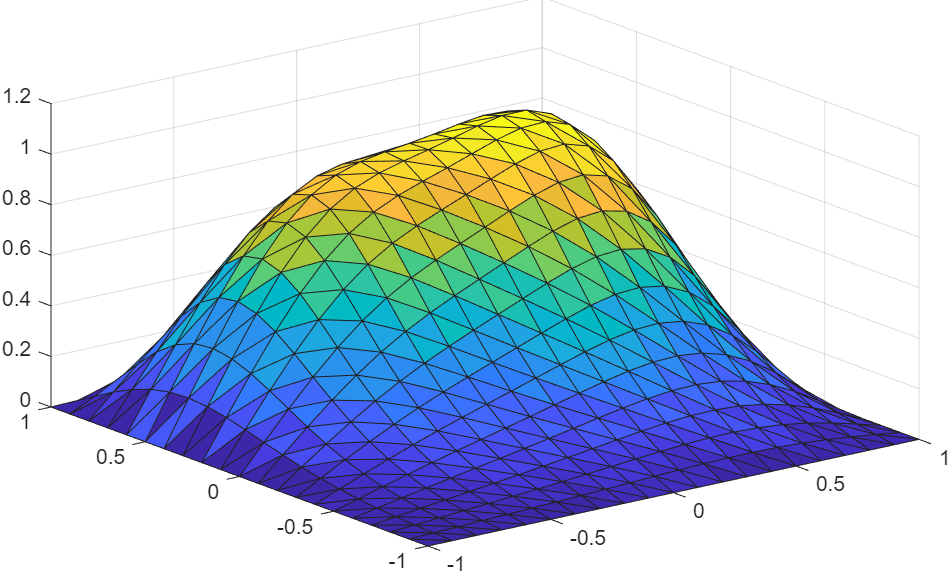}
\includegraphics[width=0.24\textwidth]{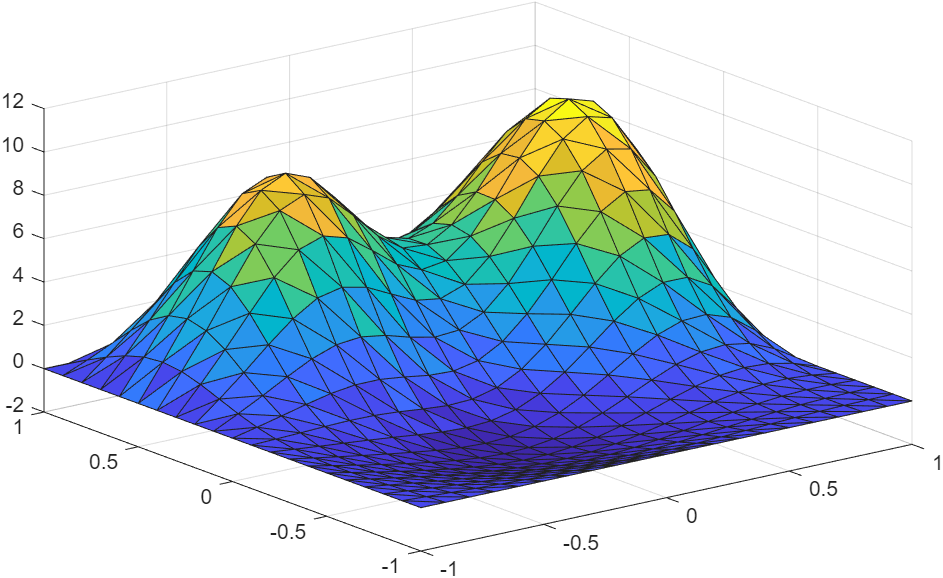}
\includegraphics[width=0.24\textwidth]{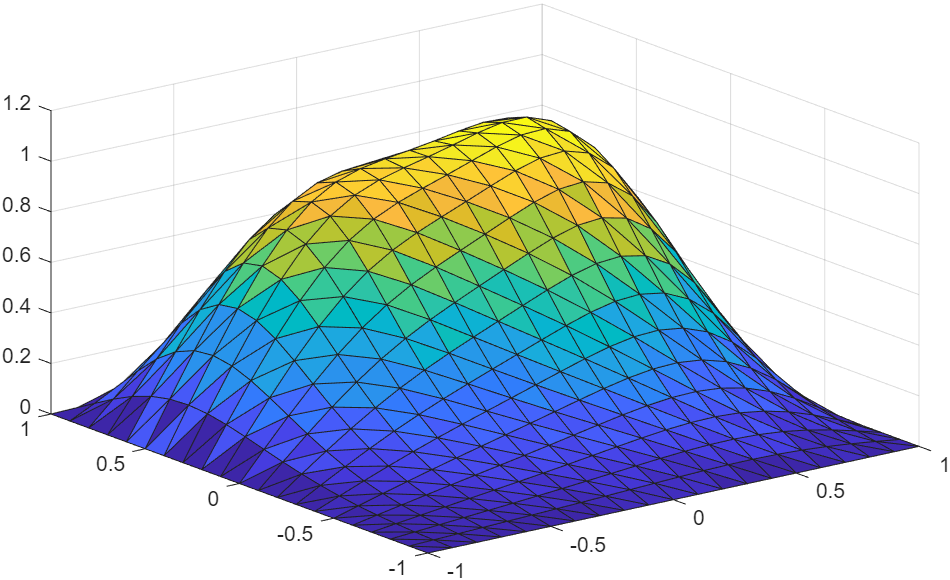}
\includegraphics[width=0.24\textwidth]{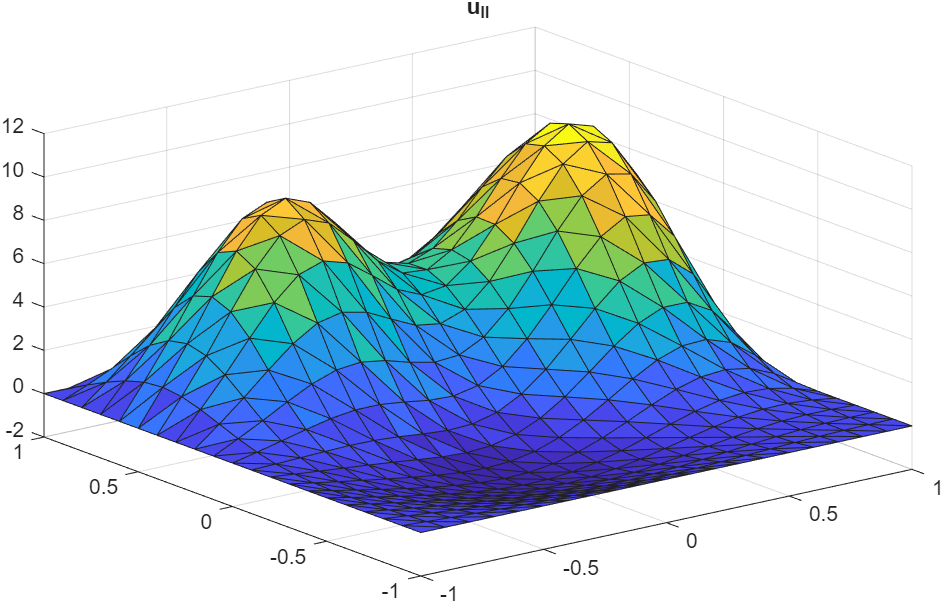}
\caption{State-control pairs for $\sigma := 10^{-2}$ and 
	$\varepsilon:=10^{-5},\ldots,10^{-8}$
	(from top to bottom);
	the first two columns show $(y_\textup{alg},u_\textup{alg})$,
	the last two columns show $(y_\textup{ll},u_\textup{ll})$.}
\label{fig:sigma_=_0.01_ii}
\end{figure}

In \cref{Table_results_1,Table_results_2}, we detail the obtained results as $\varepsilon$ decreases from  $10^{-1}$ to $10^{-8}$.
\cref{Table_results_1} documents the size of the optimality gap, 
i.e., $f(x_\textup{alg},y_\textup{alg},u_\textup{alg}) - \varphi(x_\textup{alg})$,
the violation of feasibility for \eqref{main_problem}, i.e., the shifted optimality gap $f(x_\textup{alg},y_\textup{alg},u_\textup{alg}) - \varphi(x_\textup{alg}) - \varepsilon$,
the vector $x_\textup{alg}$, the upper-level objective value $F_\textup{alg}:=F(x_\textup{alg},y_\textup{alg},u_\textup{alg})$,
and the final value $\alpha_\textup{alg}$ of the penalty parameter.

\begin{table}[H]
\centering
\begin{tabular}{cccccc}
\toprule
$\varepsilon$ & Optimality gap & Violation &  $x_\textup{alg}$ & $F_{\text{alg}}$ & $\alpha_{\text{alg}}$ \\ 
\midrule
1.0e-01   & 9.94e-02 & -6.40e-04 & $(0.294, 0.092)$ &  0.5704  & 0.4325   \\
1.0e-02   & 9.75e-03 & -2.50e-04 & $(0.292, 0.237)$ & 0.6130 & 1.2353   \\ 
1.0e-03   & 5.88e-04 & -4.12e-04 & $(0.297, 0.286)$  & 0.6288 & 4.5867   \\ 
1.0e-04   & 6.26e-05 & -3.74e-05 & $(0.302, 0.297)$ & 0.6330 & 10.0769  \\ 
1.0e-05   & 8.33e-06 & -1.67e-06 & $(0.305, 0.310)$ & 0.6383 & 37.4148  \\ 
1.0e-06   & 9.59e-07 & -4.06e-08 & $(0.303, 0.307)$ & 0.6372 & 106.8605   \\ 
1.0e-07   & 6.24e-08 & -3.76e-08 & $(0.300, 0.308)$ & 0.6370 & 234.7726  \\ 
1.0e-08   & -6.76e-08 & -7.76e-08 & $(0.300, 0.301)$ & 0.6347 & 234.7726   \\ 
\bottomrule
\end{tabular}
\caption{Computational results for different values of $\varepsilon$.}
\label{Table_results_1}
\end{table}

One of the most important observations is the steady movement of the upper-level variables toward $(0.3,0.3)$ 
which corresponds the exact global minimizer of the inverse optimal control problem, see \cite[Section~6.2]{dempe2019solving} again. 
In fact, for large values of $\varepsilon$, the result is poor. 
However, as $\varepsilon$ tends to zero, \cref{alg:stat_pen_alg} successfully identifies the target coordinates. 

The results are satisfactory for all values of $\varepsilon$ less or equal $10^{-3}$. 
We note that, for $\varepsilon := 10^{-8}$, the optimality gap is slightly negative
(which, in theory, cannot happen as the returned pair $(y_\textup{alg},u_\textup{alg})$ is feasible to \hyperref[Lower_level_problem]{(LL$(x_{\textup{alg}})$)}).
However, its value is very close to zero, so the phenomenon can reasonably be attributed to numerical tolerances.
As $\varepsilon$ decreases, the violation of feasibility for \eqref{main_problem} also gets smaller, which indicates improved feasibility of the solution
for \eqref{eq:IOC_ref} showing that the solution is becoming more accurate and reliable for the underlying inverse optimal control problem.
Note that the upper-level objective value $F_\textup{upper}$ gradually stabilizes as $\varepsilon$ approaches $0$, 
tending to approximately $0.6347$ which is close to the global minimal value 
of the inverse optimal control problem under consideration, see \cite[Figure~1]{dempe2019solving}. 
As $\varepsilon$ decreases, the final penalty parameter $\alpha_{\textup{alg}}$ grows monotonically. 
We see that a larger penalty parameter is needed to obtain more accurate solutions.
This is reasonable as the underlying inverse optimal control problem, 
which is recovered as $\varepsilon$ tends to $0$,
is rather irregular in the sense that Robinson's constraint qualification fails at all feasible points, 
see \cref{prop:RCQ}.

\cref{Table_results_2} summarizes the computational performance and tracking accuracy obtained for different values of the parameter $\varepsilon$. 
It documents the number of iterations, the total computation time, the average time \cref{alg:stat_pen_alg} needs for one iteration,
and the magnitude of change between the last two iterates for the variables $x$, $y$, and $u$, respectively.

\begin{table}[H]
\centering
\begin{tabular}{ccccccc}
\toprule
$\varepsilon$ & Iter &  Time & Time per iter& Track $x$ & Track $y$ & Track $u$ \\ 
\midrule
1.0e-01 & 40 & 1.0e+01 & 2.6e-01 & 0.00e+00& 1.92e-02& 3.04e-01 \\ 
1.0e-02 & 51 & 1.3e+01 & 2.5e-01 & 0.00e+00& 2.72e-02& 1.30e+00 \\ 
1.0e-03 & 60 & 1.5e+01 & 2.5e-01 & 0.00e+00& 4.49e-02& 5.60e-01 \\ 
1.0e-04 & 60 & 1.5e+01 & 2.5e-01 & 0.00e+00& 1.20e-01& 1.09e+00 \\ 
1.0e-05 & 70 & 1.8e+01 & 2.5e-01 & 0.00e+00& 1.75e-01& 1.53e+00 \\ 
1.0e-06 & 78 & 2.0e+01 & 2.6e-01 & 0.00e+00& 1.25e-01& 1.00e+00 \\ 
1.0e-07 & 84 & 2.1e+01 & 2.6e-01 & 0.00e+00& 6.22e-02& 8.71e-01 \\ 
1.0e-08 & 85 & 2.2e+01 & 2.6e-01 & 0.00e+00& 5.94e-02& 6.82e-01 \\ 
\bottomrule
\end{tabular}
\caption{Computational time, iteration counts, and tracking performance for different values of $\varepsilon$.}
\label{Table_results_2}
\end{table}

Speaking about computational effort, 
the number of iterations increases significantly as $\varepsilon$ gets smaller, 
from 40 iterations for $10^{-1}$ up to 85 iterations for $10^{-8}$, 
showing that tighter tolerances lead to higher computational cost. 
This increase in iterations is, however, standard in optimization when seeking high-precision results.
As $\varepsilon$ decreases, the total computation time rises from approximately 10 seconds to 22 seconds which is explained almost entirely by the increase in iteration count from 40 to 85.
The average time per iteration remains nearly constant, indicating that the computational complexity of each iteration is essentially independent of the parameter $\varepsilon$.
The fact that Track $x$ vanishes (up to machine precision) confirms that the upper-level variables have stabilized and nearly stopped changing before the algorithm is aborted.  
In contrast, Track $y$ and Track $u$ show larger values for higher $\varepsilon$ but generally decrease and stabilize as $\varepsilon$ becomes smaller, 
reflecting improved performance.

Let us note that, in order to tackle \eqref{eq:IOC_ref} computationally,
one has to solve \eqref{main_problem} as $\varepsilon$ is driven to zero.
For this purpose,
one could use a warm-starting strategy in \cref{alg:stat_pen_alg},
initializing it with the final $x$-iterate of the previous run
while reducing $\varepsilon$ at the same time.
In our experiments, we did not use this approach as we aimed to study
the dependence of \cref{alg:stat_pen_alg} from the relaxation parameter,
so warm-starting would have falsified the results.

In \cref{fig:sigma_=_0.0001}, we illustrate the output of \cref{alg:stat_pen_alg} for $\sigma:=10^{-4}$ and $\varepsilon:=10^{-8}$.
The smaller value of the regularization parameter allows the system to adjust more freely. 
This helps the state getting closer to the observed state $y_\textup{o}$, so the error is much smaller, as can be seen from the plots,
see \cref{fig:desired_state} as well.

\begin{figure}[ht]
\centering

\includegraphics[width=0.24\textwidth]{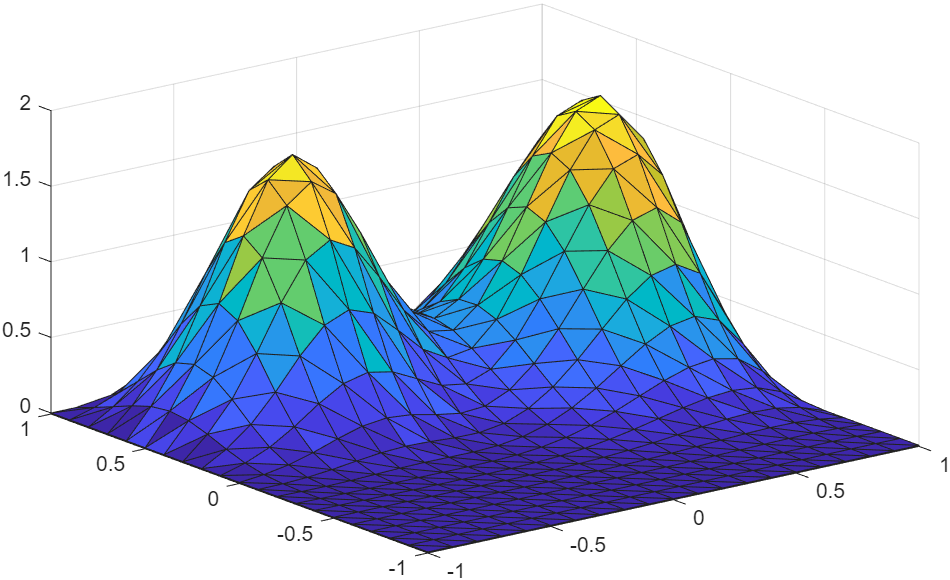}
\includegraphics[width=0.24\textwidth]{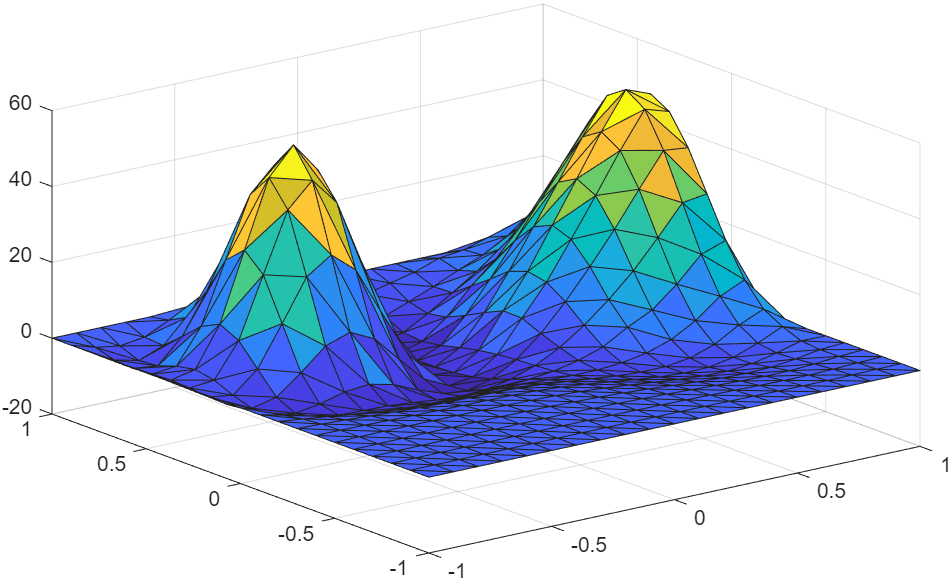}
\includegraphics[width=0.24\textwidth]{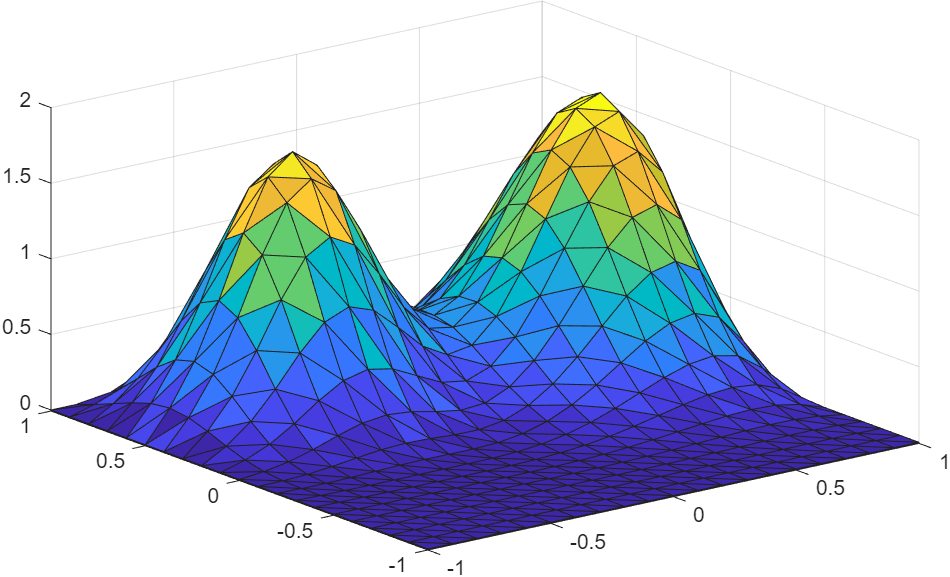}
\includegraphics[width=0.24\textwidth]{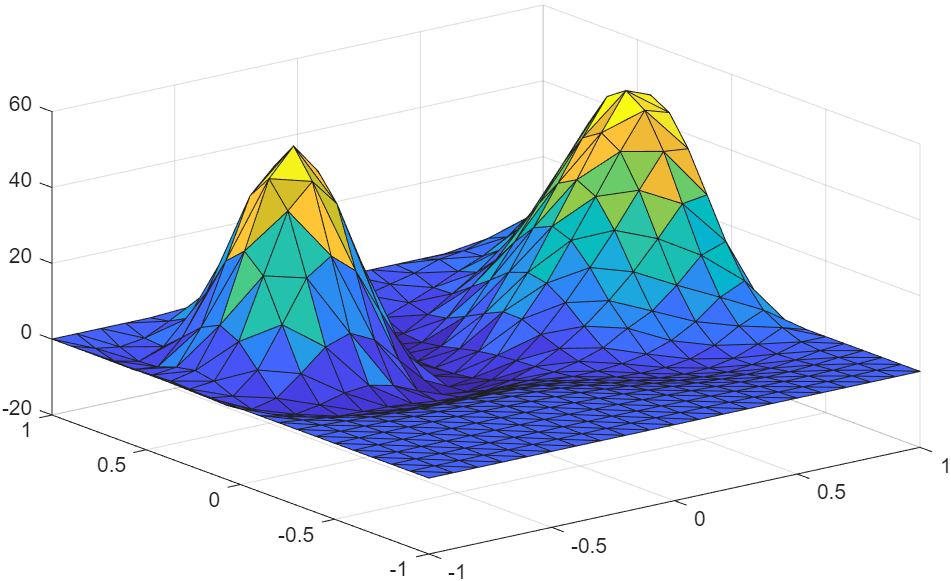}

\caption{State-control pairs for $\sigma := 10^{-4}$ and 
	$\varepsilon := 10^{-8}$;
	the first two entries show 
	$(y_\textup{alg},u_\textup{alg})$, whereas the last two show
	$(y_\textup{ll},u_\textup{ll})$.}
\label{fig:sigma_=_0.0001}
\end{figure}

\section{Conclusion}\label{sec:outro}

In this paper, we proposed a simple penalty-type solution method for the relaxed inverse optimal control problem \eqref{main_problem} and studied its convergence properties. 
The main idea of the algorithm is to identify stationary points of \eqref{main_problem} by iteratively solving the stationarity system of a penalized problem \eqref{penalized_problem}
as the penalty parameter $\alpha$ grows in some but not necessarily all iterations. 
Validity of a constraint qualification at all feasible points of \eqref{main_problem} guarantees that $\alpha$ stays bounded in this framework,
avoiding the standard difficulties associated with penalty methods. Furthermore, by making a detour via \eqref{penalized_problem}, 
we obviate to run linesearch procedures on functions or systems involving the implicitly given optimal value function of the underlying parametric optimal control problem
which would be a costly affair.
Our algorithmic approach is clearly different from the one promoted in \cite{dempe2019solving} which computes global minimizers of \eqref{eq:IOC_ref},
and could be extended to \eqref{main_problem}, as the fundamental idea behind is to exploit an iteratively refined piecewise affine upper approximation 
of the aforementioned optimal value function to obtain convex subproblems which, thus, can be solved up to global optimality.
Naturally, the global solution method from \cite{dempe2019solving} comes at high cost and can merely handle instances of \eqref{eq:IOC_ref}
where the dimension of the upper-level variables is small.
This is pretty much in contrast to our approach which comes at low cost for individual iterations as documented in \cref{sec:experiments}.

In our future research, we plan to extend the proposed solution approach to inverse optimal control problems where the upper-level decision variables are allowed to be infinite dimensional.
The analysis presented here does not cover this case,
but additional compactness assumptions on the data could be helpful to address this more general situation.
We will also study the robustness of the method in terms of inexact computations up to controllable residuals. 
Furthermore, we aim to investigate how to couple the penalty scheme with an update procedure regarding the relaxation parameter $\varepsilon$
in order to actually tackle the original inverse optimal control problem \eqref{eq:IOC_ref}. 
Recently, in \cite{kaming2025new}, it has been shown that accumulation points of a sequence of stationarity points
associated with the relaxed value function reformulation of a finite-dimensional bilevel optimization problem
as the relaxation parameter tends to zero are stationary for the original bilevel optimization problem under a so-called asymptotic constraint qualification.
Noting that asymptotic constraint qualifications can reasonably be extended to optimization problems in Banach spaces, 
see \cite{BoergensKanzowMehlitzWachsmuth2020} for a recent study, there is some reasonable hope that the results from \cite{kaming2025new}, at least partially,
can be transferred to the setting of inverse optimal control.

\bibliographystyle{habbrv}
\bibliography{Bibliography}

\appendix

\end{document}